\documentclass{mylms}
\usepackage{amscd}
\usepackage{amssymb}
\usepackage{cite}
\usepackage{enumerate}
\usepackage{amsmath}
\usepackage[usenames, dvipsnames]{color}
\usepackage{hyperref}
\hypersetup{
    colorlinks=true,
  linkcolor=red,
   filecolor=blue,
    urlcolor=green,
    citecolor=blue
}
\usepackage{bookmark}

\def\w*lim{\mathop{\mbox{\textup{w*-lim}}}}

\usepackage{mathrsfs}

\newtheorem{theorem}{\sc \textbf{Theorem}}[section]
\newtheorem{lemma}[theorem]{\sc \textbf{Lemma}}
\newtheorem{corollary}[theorem]{\sc \textbf{Corollary}}
\newtheorem{proposition}[theorem]{\sc \textbf{Proposition}}
\newtheorem{remark}[theorem]{\sc \textbf{Remark}}
\newtheorem{definition}[theorem]{\sc \textbf{Definition}}

\newtheorem{question}[theorem]{\sc \textbf{Question}}

\newtheorem{notation}[theorem]{\textbf{Notation}}

\newtheorem{lem}[theorem]{\sc \textbf{Lemma}}

\newcommand{\cH}{{\mathcal H}}

\newcommand{\cM}{{\mathcal M}}
\newcommand{\cN}{{\mathcal N}}

\newcommand{\cP}{{\mathcal P}}

\newcommand{\bC}{{\mathbb C}}

\newcommand{\bM}{{\mathbb M}}

\newcommand{\bR}{{\mathbb R}}

\newcommand{\bT}{{\mathbb T}}

\newcommand{\bZ}{{\mathbb Z}}

\newcommand{\sX}{{\mathscr X}}
\newcommand{\sY}{{\mathscr Y}}
 \usepackage{relsize}
\newcommand{\fD}{{\mathfrak D}}
\newcommand{\fM}{{\mathfrak M}}

\title[Operator $\theta$-H\"{o}lder functions]{Operator $\theta$-H\"{o}lder functions with respect to $\left\|\cdot\right\|_p$, $0< p\le \infty$}

\author{J. Huang, F. Sukochev and D. Zanin}

\extraline{Received 12 December 2019; revised 20 April 2021.}

\classno{47A55, 46L51, 47A60, 47B10.}

\begin{document}

\maketitle
\begin{abstract}
Let $\theta \in(0,1)$ and $(\cM,\tau)$ be a semifinite von Neumann algebra.
We consider the function spaces  introduced by Sobolev \cite{Sobolev,Sobolev17} (denoted by $S_{d,\theta}$),
showing  that there exists a constant $d>0 $ depending on $p$, $0<p\le \infty$,  only such that
 every  function  $f:\bR\rightarrow \bC \in S_{d,\theta}$
is  operator $\theta$-H\"older with respect to $\left\|\cdot \right\|_p$, that is,  there exists a constant $C_{p,f}$ depending on $p$ and $f$ only such that the estimate
$$\left\|f(A) -f(B)\right\|_p \le C_{p,f}\left\| \left| A-B \right|^\theta \right \|_p $$
holds
for arbitrary self-adjoint $\tau$-measurable operators $A$ and $ B$.
In particular, we obtain a sharp condition such that  a  function $f$ is  operator $\theta$-H\"older with respect to all quasi-norms $\left\|\cdot \right\|_p$,  $0<p\le \infty$,
which
 complements the results on the case for  $ \frac1\theta  < p<\infty $ by Aleksandrov and Peller \cite{AP2}, and the case  when $p=\infty$ treated  by  Aleksandrov and Peller \cite{AP3}, and by  Nikol$^\prime$skaya and Farforovskaya \cite{NF}.
 As an application, we show that this class of functions is operator $\theta$-H\"older with respect to a wide class of symmetrically quasi-normed operator spaces affiliated with $\cM$,  which  unifies the results on specific functions due to Birman, Koplienko  and  Solomjak \cite{BKS,BS89},  Bhatia \cite{Bhatia},  Ando \cite{Ando},  and Ricard   \cite{Ricard, Ricard2} with significant extension.
In addition,
when $\theta>1$,
we
obtain a reverse of the  Birman--Koplienko--Solomjak inequality, which extends a couple of existing results on fractional powers $t\mapsto t^\theta$ by Ando et al.
\end{abstract}

\section{Introduction}\label{intro}
Let $0<\theta<1$.
A function $f:\bR \rightarrow \bC$ is called a \emph{$\theta$-H\"older function} if it satisfies the inequality
\begin{align}\label{conH}
|f(x)-f(y)|\le {\rm const}  |x-y|^\theta, ~x,y\in \bR.
\end{align}
  Let $\cM$ be a semifinite von Neumann algebra equipped with a semifinite faithful normal trace $\tau$.
  We denote by $S(\cM,\tau)$ the $*$-algebra  of all $\tau$-measurable operators affiliated with $\cM$.
The main object of the present paper is the so-called  \emph{operator $\theta$-H\"older} (or \emph{operator-H\"older of order $\theta$} \cite{ST}) functions  $f$  with respect to
$\left\|\cdot \right\|_p$, where $\left\|\cdot \right\|_p$ stands for the quasi-norm of
the
 non-commutative $L_p$-space $L_p(\cM,\tau)$, $0< p\le \infty$ \cite{PX}.
 That is, for a  given $p>0$, there exists a constant $C(p,f)$ depending on $p$ and $f$ (and, obviously, on $\theta$) 
  only such that
\begin{align}\label{def:opH}
\|f(X)-f(Y)\|_{p} \le C(p,f) \left\|\left|X-Y \right|^
\theta \right\|_p
\end{align}
for an arbitrary semifinite von Neumann algebra $\cM$ and any self-adjoint operators $X,Y\in S(\cM,\tau)$.

The principal results of this paper were motivated by the following question:
%
 \begin{question} \label{Q:1.1}
 Let $\theta\in (0,1)$.
 What is the class of functions $f:\bR \rightarrow \bC$ such that for an arbitrary   $p\in (0,\infty ] $, there exists  a constant  $ C(p,f)$ depending on $p$ and $f$ only (independent of $\cM$)  such that
 $$\|f(X) -f(Y)\|_p \le C(p,f)\left\||X-Y|^\theta \right\|_p, ~\forall X=X^*,Y=Y^*\in  S(\cM,\tau)?$$
 \end{question}

This question has a deep history in operator theory.
The starting point is the so-called Powers-St${\o}$rmer inequality \cite{Powers_S} (see also \cite{Bhatia,Hemmen_A,Kittaneh}):
 $$
\left\|X^\frac 12 -Y^\frac 12 \right\|_2 \le \left\||X-Y|^\frac 12 \right\|_2, ~X,Y \in B(\cH)^+,
$$
where $B(\cH)^+$ stands for the positive part of the  $*$-algebra $B(\cH)$ of all bounded linear operators on a Hilbert space $\cH$.
A remarkable extension of the above inequality  is due to  Birman, Koplienko  and  Solomjak by using double operator integrals techniques \cite{BKS} (an alternative proof was given in \cite{PS1}):
 if $0<\theta<1$, $f(t):=|t|^\theta$, $t\in \bR$ and  $\left\|\cdot\right\|$ is an arbitrary fully symmetric norm on $B(\cH)$, then
 \begin{align}\label{eq:1}
\left\|X^\theta -Y^\theta \right\| \le \left\||X-Y|^\theta \right\|, ~X,Y \in B(\cH)^+.
\end{align}

After that, many mathematicians enlarged the classes of functions $f$ or the (quasi)-norms for which \eqref{eq:1} holds.
 There is a vast literature concerned with problems of this type, with a large number of deep results (see e.g. \cite{AP1,AP2,AP3,AP4,APPS,Ando,PS2,Ricard,DD1995,CPSZ,KPSS,AP11,AZ}).
In particular, Ando \cite{Ando} replaced the function $t\rightarrow t^\theta$ with
  any non-negative \emph{operator monotone function} $f:\bR^+ \to \bR^+$ (classic examples are $t\mapsto t^\theta$, $0< \theta<1$, $t\mapsto \log(t+1)$ and $t\mapsto \frac{t}{r+t}$, $r>0$),  showing that
$$
\left\|f(X)-f(Y)\right\| \le \left\| f(|X-Y|) \right\|, ~X,Y\in B(\cH)^+ ,
$$
for any fully symmetric norm $\left\|\cdot\right\|$ on $B(\cH)$.
This result  was later extended by P. Dodds and T. Dodds \cite{DD1995} to the case of fully symmetric spaces affiliated with a semi-finite von Neumann algebra (see also \cite{AZ,Ando_H,Hemmen_A} for related topics on operator monotone functions).
Kosaki \cite{Hiai_N} proved \eqref{eq:1}   for the Haagerup $L_p$-spaces, $p\ge 1$ (see also \cite[Proposition 7]{Kosaki1984}).

If $\theta=1$,
then operator $\theta$-H\"older functions are the so-called \emph{operator-Lipschitz} functions.
It is well-known \cite{Far,Kato} that a Lipschitz function on the real line is not necessarily \emph{operator-Lipschitz} with respect to $\left\|\cdot\right\|_\infty$, i.e., the condition
$$|f(x) -f(y)|\le {\rm const} |x-y|, ~x,y\in \bR,$$
does not imply the that estimate
$$\|f(A)-f(B)\|_\infty \le {\rm const}' \|A-B\|_\infty$$
holds for any   self-adjoint operators $A,B\in B(\cH)$.
Davies \cite{Davies} showed that the absolute value function $t\mapsto |t|$ is not \emph{operator-Lipschitz} for $\left\|\cdot\right\|_1$.
However, it has been shown recently by Potapov and Sukochev  \cite{PS2}   that  every Lipschitz function is necessarily operator-Lipschitz in the Schatten--von Neumann ideal  $S_p$ (or noncommutative $L_p$-space) for any $1<p<\infty$ (see \cite{CMPS} for the best constant for operator Lipschitz functions).
A class of operator-Lipschitz functions for Schatten--von Neumann ideal $S_p$ for $p<1$ has been identified in \cite{MS20}.

An interesting  extension of \eqref{eq:1} was obtained recently by Aleksandrov and Peller \cite{AP2,AP3} (the case for $\left\|\cdot\right\|_\infty$ was obtained also by Nikol$^\prime$skaya and Farforovskaya  \cite{NF}).
They showed that the situation changes dramatically if we consider $\theta$-H\"older functions instead of Lipschitz functions, i.e., for any $0<\theta <1$,
  condition \eqref{conH} 
  implies that
for every  $1< p\le \infty$ and any self-adjoint $X,Y\in B(\cH)$, one has
$$\|f(X)-f(Y)\|_{p/\theta } \le C_{p,f} \|X-Y\|_p^\theta,  $$
with  $C_{p,f}$ depending on $p$ and $f$ only (see \cite{APPS} for the case of normal operators in $\left( B(\cH), \left\|\cdot \right\|_\infty \right)$).
Equivalently, every $\theta$-H\"older function $f$ on $\bR$ is necessarily operator $\theta$-H\"older  with respect to $\left\|\cdot\right\|_p$ provided that  $p>\frac1\theta$, that is,
$$ \|f(X)-f(Y)\|_p \le C_{p,f}\left\|| X-Y|^\theta\right\|_p. $$
However,
a $\theta$-H\"older   function $f$   is not necessarily  operator $\theta$-H\"older with respect to $\left\|\cdot\right\|_p$,
when $p=\frac1\theta$ (see \cite[Section 9]{AP2}).
Very little is known for the case of $\left\|\cdot\right\|_p$ when $p\le \frac1\theta$, and  for other symmetric quasi-norms.
Birman and Solomjak \cite[Proposition 4.9]{BS89} obtained a result for weak $L_p$-spaces, showing  that for every  $0< p< \infty$, $0<\theta<1$ and any positive  $X,Y\in B(\cH)$, one has
$$\left\| X^\theta-Y^\theta \right\|_{p,\infty } \le C_{p,\theta} \left\| |X-Y|^\theta  \right\|_{p,\infty },  $$
with  $C_{p,\theta}$ depending on $p$ and $\theta$ only.

For the general case of quasi-norms $\left\|\cdot\right\|_p$, $p>0$,
 a weak-type result (in terms of quasi-commutator estimates) was recently obtained by  Sobolev  \cite{Sobolev}.
  He introduced a  class (of $\theta$-H\"{o}lder functions)
$$S_{d,\theta} := \Big\{f\in C^d(\bR \backslash \{0\})\cap C(\bR): ~ \|f\|_{S_{d,\theta}} :=\max_{0\le k \le d } \sup _{x\ne 0}|f^{(k)}(x)||x|^{-\theta +k}<\infty \Big\},$$
which also plays a crucial  role in the study of Wiener--Hopf operators \cite{Sobolev17,LSS}.
Note that all functions in $S_{d,\theta}$ are $\theta$-H\"{o}lder when $d\ge 1$ (see \cite[(2.3)]{Sobolev17} or
Lemma \ref{Lemmaholdertrivial}).
Sobolev considered $f\in S_{d,\theta},$  supported on a compact  interval $[-r,r]$ and proved the inequality
\begin{align}\label{eq:so}
\|f(A)-f(B)\|_\mathfrak{S} \le C_{d,f,\sigma,\theta,r}\left\| \left|A-B\right|^\sigma\right\|_\mathfrak{S},\quad A=A^{\ast},B=B^{\ast}\in B(\cH)
\end{align}
for $\sigma<\theta$ such that $(d-\sigma)^{-1} <p \le 1$ and for every symmetrically $p$-normed ideal $(\mathfrak{S},\|\cdot\|_{\mathfrak{S}})$ of $B(\cH)$.  However, since $\sigma$ is strictly smaller than $\theta,$  $C_{d,f,\sigma,\theta,r}\rightarrow \infty$ as $\sigma\rightarrow\theta,$ and since the function
 $f:\bR \rightarrow \bR, ~ t \mapsto |t|^\theta$  has unbounded support, until recently it remained unknown  whether the classic $\theta$-H\"older function $t \mapsto |t|^\theta$ is operator $\theta$-H\"older (i.e. \eqref{def:opH}) even for Schatten $p$-class, $0<p<1$.

A recent breakthrough is due to
Ricard \cite{Ricard}, who established \eqref{eq:1} for  noncommutative $L_p$-spaces affiliated with $\cM$, $0<p\le \infty$.
Precisely, Ricard \cite{Ricard} proved that there exists a constant $C_{p,\theta}$ depending on $p$ and $\theta$ only such that
\begin{align}\label{eq:2}
\left\||X|^\theta -|Y|^\theta \right\|_p \le C_{p,\theta } \left\||X- Y|^{\theta}\right \|_p, ~X=X^*,Y=Y^*\in S(\cM,\tau)
\end{align}
 (see also \cite{DD1995,Ando} for the case when $p\ge 1$).
This result demonstrates that
$f:t\mapsto|t|^\theta$ provides a non-trivial positive (partial) answer to
Question \ref{Q:1.1}.
However, the argument in \cite{Ricard} heavily rests on the homogeneity of fractional power functions, which does not seem to extend to more general functions.

For the sake of convenience, we denote  $$S_{\infty,\theta} := \cap_{0\le d<\infty} S_{d,\theta} .$$
In particular, all functions in $S_{\infty,\theta}$ are $\theta$-H\"{o}lder (note that $\sup_{d\ge 0}\left\| f\right\|_{S_{d,\theta}}$ is not necessarily finite).
It is immediate that  $t\mapsto |t|^\theta \in S_{ \infty,\theta}$.
We show that the class $S_{\infty,\theta}$ generated by the function spaces $S_{d,\theta}$ introduced in \cite{Sobolev} provides a  condition such that Question \ref{Q:1.1} has an affirmative answer.
 The following is the main result of the present paper,
which complements the results in \cite{AP2,AP3} (see also \cite{NF} and \cite{APPS}; note that these results do not require differentiability imposed on the functions) for  Schatten $p$-classes when $p> \frac1\theta$,  and extends those in  \cite{Ricard,Ricard2,BKS,BS89}.
\begin{theorem}
Let $\theta\in (0,1)$.
Then, for every $p>0$, there exists  a constant  $d=d(p)$ such that every function $f\in S_{d,\theta}$ is  operator $\theta$-H\"older with respect to $\left\|\cdot \right\|_p$.
In particular,
 every function   $f \in S_{\infty,\theta}$
 is  operator $\theta$-H\"older with respect to $\left\|\cdot \right\|_p$ for  arbitrary $p>0$, that is, for $p>0$ and $0<\theta<1$, there exists a constant $C_{p,\theta}$ such that
\begin{align}\label{eq:4}
\left\|f(X) -f(Y)\right\|_{p} \le C_{p, \theta } \left\|f \right\|_{S_{d,\theta}} \left\||X- Y|^{\theta}\right \|_{p} ~X=X^* ,Y=Y^* \in S(\cM,\tau).
\end{align}
\end{theorem}

As an application of \eqref{eq:4},
  we obtain a submajorization inequality for general self-adjoint $\tau$-measurable operators $X $ and $Y$, i.e., there exists  a  constant $C_{p,\theta}$ depending only on $p$ and $\theta$ such that for any $X=X^*, Y=Y^*\in S(\cM,\tau)$, we have  (see Theorems \ref{th:subm} and \ref{th:mainE})
\begin{align}\label{intro_subm}
\left(\mu\left( f(X) -f(Y) \right) \right)^p  \prec \prec C_{p,\theta}\|f\|_{S_{d,\theta}}^p \cdot \mu(|X-Y|^\theta)^p.
\end{align}
Here, $\prec\prec$ denotes submajorisation in the sense of Hardy--Littlewood--Polya and $\mu(\cdot)$ is the generalised singular value function \cite{DPS,LSZ}.
The estimate \eqref{intro_subm} extends results in \cite{BKS,Ando,DD1995}.

The majority
of Banach symmetric spaces used in analysis are fully symmetric rather than just symmetric.
 Moreover, a wide class of symmetric $p$-normed spaces can be constructed from  the $1/p$-th power of   symmetriclly normed spaces \cite{Kalton_S,DDS2014} (e.g., $\left\|\cdot\right\|_{p,\infty}$ is equivalent to the $
\frac{1}{p+\varepsilon}$-power of a fully symmetric norm for any $\varepsilon>0$ \cite[Chapter 4, Lemma 4.5]{Bennett_S}).
Denote by $E(\cM,\tau)^{(p)} $ the $1/p$-power of a fully symmetric space  $E(\cM,\tau)$ affiliated with $\cM$ \cite{DDS2014}.
That is, $E(\cM,\tau)^{(p)} $ is the class of operators $X\in S(\cM,\tau)$ with $|X|^{p}\in E(\cM,\tau)$ and $\left\|X \right\|_{E^{(p)}}  = \left\| |X|^p \right\|_E^{1/p}$.
We obtain the following corollary.
\begin{corollary}\label{intr_main}
Let $\theta\in (0,1)$ and $p\in (0,\infty)$. Then, there exists a constant  $d=d(p)$ such that  every function   $f \in S_{d,\theta} $
is operator $\theta$-H\"older with respect to $\left\|\cdot \right\|_{E^{(p)}}$, i.e.,
there exists a constant $C_{p,\theta}$ such that
for any  self-adjoint $X,Y\in S(\cM,\tau)$ with  $X-Y\in E^{(p)}(\cM,\tau)$, we have
\begin{align}\label{ineq:fep}
\left\|f(X) -f(Y)\right\|_{E^{(p)}} \le C_{p, \theta } \|f\|_{S_{d,\theta}} \left\||X- Y|^{\theta}\right \|_{E^{(p)}} .
\end{align}
\end{corollary}

This corollary  extends results in \cite{BKS}, \cite{BS89} \cite{Ricard2} and \cite{Ricard} in two directions.
Firstly, instead of considering only noncommutative $L_p$-spaces,
we prove that \eqref{eq:2} holds true for a very wide class of quasi-Banach symmetric spaces.
 Secondly, we extend significantly the class of functions $f$ for which the result is applicable.
 In particular, letting $f(t)=|t|^{\theta}$ and $p=1$ in \eqref{ineq:fep},  we obtain \eqref{eq:1}; letting $E(\cM,\tau) =L_1(\cM,\tau)$, we obtain  \eqref{eq:2}.
Indeed, the space $S_{
\infty,\theta}$ embraces a wide class of functions, e.g.,
   Schwartz class functions and  classic operator monotone functions such as $t\mapsto {\rm sgn}(t)\log(|t|+1)$ and $t\mapsto \frac{t}{r+|t|}$, $t>0$.
As an application of \eqref{intro_subm}, we obtain the inequality for  fractional powers when  $\theta >1$ (see Corollaries \ref{cor:inv} and \ref{cor:>1}), extending existing results for fully symmetric ideals on $B(\cH)$ in \cite{Bhatia,Ando,Hemmen_A}.

Finally, we obtain the estimates for  absolute value map, commutators and quasi-commutators, extending results in \cite{Bhatia1988}, \cite{Ricard} and \cite{Kittaneh_K,Kosaki1984}.
In contrast with the recent result by Sobolev \cite[Theorem 2.4]{Sobolev},
the compactness of the supports of the functions imposed in \eqref{eq:so} is no longer required 
and
 we can also treat  the case when $\theta =\sigma$.
Comparing with the constant $d$ obtained in \cite{Sobolev}, the constant $d(p)$ obtained in our paper is sharp (see Remark \ref{7.4}),   which nicely complements \cite[Theorem 2.4]{Sobolev}.
The main result presented in this paper  could be used to provide an alternative way to prove  main results in \cite{Sobolev17} and \cite{LSS} even for unbounded domains $\Lambda$ in $\bR^d$.
To keep this paper to a reasonable length, this application will be published separately.

\section{Preliminaries}\label{s:p}

In this section, we recall main notions from the theory of noncommutative integration and  recall some properties of the  generalised singular value function.
In what follows,  $\cH$ is a  Hilbert space and $B(\cH)$ is the
$*$-algebra of all bounded linear operators on $\cH$, and
$\mathbf{1}$ is the identity operator on $\cH$.
Let $\mathcal{M}$ be
a von Neumann algebra represented on $\cH$.
For details on von Neumann algebra
theory, the reader is referred to e.g. \cite{Dixmier}.
General facts concerning measurable operators may
be found in \cite{Nelson}, \cite{Se} (see also 
the forthcoming book \cite{DPS}).
\subsection{$\tau$-measurable operators and generalised singular values}

A linear operator $X:\mathfrak{D}\left( X\right) \rightarrow \cH $,
where the domain $\mathfrak{D}\left( X\right) $ of $X$ is a linear
subspace of $\cH$, is said to be {\it affiliated} with $\mathcal{M}$ (denoted by $X\eta \cM$)
if $YX\subseteq XY$ for all $Y\in \mathcal{M}^{\prime }$, where $\mathcal{M}^{\prime }$ is the commutant of $\mathcal{M}$.
A linear
operator $X:\mathfrak{D}\left( X\right) \rightarrow \cH $ is termed
{\it measurable} with respect to $\mathcal{M}$ if $X$ is closed,
densely defined, affiliated with $\mathcal{M}$ and there exists a
sequence $\left\{ P_n\right\}_{n=1}^{\infty}$ in the lattice of all
projections of $\mathcal{M}$, $\cP\left(\mathcal{M}\right)$, such
that $P_n\uparrow \mathbf{1}$, $P_n(\cH)\subseteq\mathfrak{D}\left(X\right) $
and $\mathbf{1}-P_n$ is a finite projection (with respect to $\mathcal{M}$)
for all $n$. It should be noted that the condition $P_{n}\left(
\cH\right) \subseteq \mathfrak{D}\left( X\right) $ implies that
$XP_{n}\in \mathcal{M}$. The collection of all measurable
operators with respect to $\mathcal{M}$ is denoted by $S\left(
\mathcal{M} \right) $, which is a unital $\ast $-algebra
with respect to strong sums and products (denoted simply by $X+Y$ and $XY$ for all $X,Y\in S\left( \mathcal{M}\right) $).

Let $X$ be a self-adjoint operator affiliated with $\mathcal{M}$.
We denote its spectral measure by $\{E^X\}$.
It is well known that if
$X$ is a closed operator affiliated with $\mathcal{M}$ with the
polar decomposition $X = U|X|$, then $U\in\mathcal{M}$ and $E\in
\mathcal{M}$ for all projections $E\in \{E^{|X|}\}$.
Moreover,
$X\in S(\mathcal{M})$ if and only if $X$ is closed,
 densely
defined, affiliated with $\mathcal{M}$ and $E^{|X|}(\lambda,
\infty)$ is a finite projection for some $\lambda> 0$.
 It follows
immediately that in the case when $\mathcal{M}$ is a von Neumann
algebra of type $III$ or a type $I$ factor, we have
$S(\mathcal{M})= \mathcal{M}$.
For type $II$ von Neumann algebras,
this is no longer true.
From now on, let $\mathcal{M}$ be a
semifinite von Neumann algebra equipped with a faithful normal
semifinite trace $\tau$.

For any closed and densely defined linear operator $X:\mathfrak{D}\left( X\right) \rightarrow \cH $,
the \emph{null projection} $n(X)=n(|X|)$ is the projection onto its kernel $\mbox{Ker} (X)$,
 the \emph{range projection } $r(X)$ is the projection onto the closure of its range $\mbox{Ran}(X)$ and the \emph{support projection} $s(X)$ of $X$ is defined by $s(X) ={\bf{1}} - n(X)$.

An operator $X\in S\left( \mathcal{M}\right) $ is called $\tau$-measurable if there exists a sequence
$\left\{P_n\right\}_{n=1}^{\infty}$ in $P\left(\mathcal{M}\right)$ such that
$P_n\uparrow \mathbf{1},$ $P_n\left( \cH \right)\subseteq \mathfrak{D}\left(X\right)$ and
$\tau(\mathbf{1}-P_n)<\infty $ for all $n$.
 The collection $S\left(\cM, \tau \right) $ of all $\tau $-measurable
operators is a unital $\ast $-subalgebra of $S\left(
\mathcal{M}\right) $ denoted by $S\left( \mathcal{M}, \tau\right)
$.
We denote by $S(\cM,\tau)_h$ the real subspace of self-adjoint elements of $S(\cM,\tau)$.
It is well known that a linear operator $X$ belongs to $S\left(
\mathcal{M}, \tau\right) $ if and only if $X\in S(\mathcal{M})$
and there exists $\lambda>0$ such that $\tau(E^{|X|}(\lambda,
\infty))<\infty$. Alternatively, an unbounded operator $X$
affiliated with $\mathcal{M}$ is  $\tau$-measurable
\cite{FK} if and only if
$\tau\left(E^{|X|}\bigl(n,\infty\bigr)\right)\rightarrow 0$ as $ n\to\infty.$
\begin{definition}\label{mu}
Let a semifinite von Neumann  algebra $\mathcal M$ be equipped
with a faithful normal semi-finite trace $\tau$ and let $X\in
S(\mathcal{M},\tau)$. The generalised singular value function $\mu(X):t\rightarrow \mu(t;X)$ of
the operator $X$ is defined by setting
$$
\mu(s;X)
=
\inf\{\|XP\|_\infty :\ P=P^*\in\mathcal{M}\mbox{ is a projection,}\ \tau(\mathbf{1}-P)\leq s\}.
$$
\end{definition}
An equivalent definition in terms of the
distribution function of the operator $X$ is the following. For every $X\in S(\mathcal{M},\tau)_h$, we define a right-continuous function $d_X(\cdot)$ by
$d_X(t)=\tau(E^{X}(t,\infty))$, $ t>0$ (see e.g.  \cite{FK}).
We have (see e.g. \cite{FK} and \cite{LSZ})
$$
\mu(t; X)=\inf\{s\geq0:\ d_{|X|}(s)\leq t\}.
$$


For every $\varepsilon,\delta>0,$ we define the set
$$V(\varepsilon,\delta)=\{x\in S(\mathcal{M},\tau):\ \exists P \in \cP\left(\mathcal{M}\right)\mbox{ such that }
\|X(\mathbf{1}-P)\|_\infty \leq\varepsilon,\ \tau(P)\leq\delta\}.$$ The topology
generated by the sets $V(\varepsilon,\delta)$,
$\varepsilon,\delta>0,$ is called the measure topology $t_\tau$ on $S(\cM,\tau)$ \cite{DPS, FK, Nelson}.
A further important vector space topology on $S(\cM,\tau)$ is the \emph{local measure topology} \cite{DP2,DPS}.
A neighbourhood base for this topology is given by the sets $V(\varepsilon, \delta; P)$, $\varepsilon, \delta>0$, $P\in \cP(\cM) $ with $\tau(P)<\infty$, where
$$V(\varepsilon,\delta; P ) = \{X \in S(\cM,\tau): PXP \in V(\varepsilon,\delta)\}. $$
If $\{X_\alpha\}\subset S(\cM,\tau)$ is a net and if $X_\alpha \rightarrow_\alpha X \in S(\cM,\tau)$ in local measure topology, then $X_\alpha Y\rightarrow  XY  $ and $YX _\alpha \rightarrow YX $ in the local measure topology for all $Y \in S(\cM,\tau)$ \cite{DP2,DPS}.

\subsection{Quasi-Banach symmetric spaces}\label{sub:sys}It is  convenient to recall the notion of a quasi-normed space \cite{Kalton,KPR}.
Let $X$ be a linear space.
A strictly positive, absolutely homogeneous functional $\left\|\cdot\right\|: X\rightarrow [0,\infty)$ is called a quasi-norm if there exists a constant $K>0$  such that
$$\|x_1+x_2\|\le K\left(\|x_2\| +\|x_2\| \right) , ~x_1,x_2\in X.$$
The optimal choice of $K$ will be called the  modulus of concavity of the quasi-norm.
\begin{definition}\label{opspace}
Let a semifinite von Neumann  algebra $\cM$ be equipped
with a faithful normal semi-finite trace $\tau$.
Let $E(\cM,\tau)$ be a linear subset in $S({\mathcal{M}, \tau})$
equipped with a quasi-norm $\left\|\cdot\right\|_{E}$.
We say that
$E(\cM,\tau)$ is a \textit{symmetrically quasi-normed  space}  if
for $X\in
E(\cM,\tau)$, $Y\in S({\mathcal{M}, \tau})$ and  $\mu(Y)\leq \mu(X)$ imply that $Y\in E(\cM,\tau)$ and
$\|Y\|_E\leq \|X\|_E$.
In particular, if $\left\|\cdot\right\|_ E$ is a norm, then $ E(\cM,\tau) $ is called a symmetrically normed space.
\end{definition}

A symmetrically (quasi-)normed space is called a (quasi-Banach) symmetric space if it is complete.
It is well-known  that any quasi-symmetrically normed space $E(\cM,\tau)$ is a quasi-normed $\cM$-bimodule, that is $AXB\in E(\cM,\tau)$ for any $X\in E(\cM,\tau)$, $A,B\in \cM$ and $\left\|AXB\right\|_E\leq \left\|A\right\|_\infty\left\|B\right\|_\infty\left\|X\right\|_E$ \cite{Sukochev,DP2,DPS}.

If $X,Y\in S(\cM,\tau)$, then $X$ is said to be submajorized by $Y$, denoted by $X\prec\prec Y$ (Hardy--Littlewood--Polya submajorization, see \cite{LSZ,DPS,DP2}), if
\begin{align*}
\int_{0}^{t} \mu(s;X) ds \le \int_{0}^{t} \mu(s;Y) ds, ~\forall t>0.
\end{align*}

 A symmetric space $E(\cM,\tau)\subset S(\cM,\tau)$ is called \emph{strongly symmetric} if its norm $\left\|\cdot\right\|_E$ has the additional property that   $\left\|X\right\|_E \le \left\|Y\right\|_E$ whenever $X,Y \in E(\cM,\tau)$ satisfy $X\prec\prec Y$.
In addition, if $X\in S(\cM,\tau)$, $Y \in E(\cM,\tau)$ and $X\prec\prec Y$ imply that $X\in E(\cM,\tau)$ and $\left\|X\right\|_E \le \left\|Y\right\|_E$, then
 $E(\cM,\tau)$  is called \emph{fully symmetric space} (of $\tau$-measurable operators).
We denote by $E(\cM,\tau)_h$ the real subspace of a symmetrically quasi-normed space $E(\cM,\tau)$.

A wide class of quasi-Banach symmetric operator spaces associated with the von  Neumann algebra $\cM$ can be constructed from concrete quasi-Banach symmetric function spaces studied extensively in e.g. \cite{KPS}.
Let $(E(0,\infty),\left\|\cdot\right\|_{E(0 ,\infty)})$ be a (quasi-Banach) symmetric function space on the semi-axis $(0,\infty)$.
That is, a quasi-Banach symmetric space for the von Neumann algebra $L_\infty(0,\infty)$ with trace given by Lebesgue integration.
The pair
$$E(\cM,\tau)=\{X\in S(\cM,\tau):\mu(X)\in E(0,\infty)\},\quad \left\|X\right\|_{E(\cM,\tau)}:=\left\| \mu(X) \right\|_{E(0,\infty)}$$ is a (quasi-Banach) symmetric operator space affiliated with $\cM$  (see e.g.  \cite{Kalton_S,LSZ,HLS2017,Sukochev}).
For convenience, we denote $\left\|\cdot\right\|_{E(\cM,\tau)}$ by $\left\|\cdot\right\|_E$.
Many properties of quasi-Banach symmetric spaces, such as reflexivity, Fatou property, order continuity of the norm as well as K\"{o}the duality carry over from commutative symmetric function space $ E(0,\infty)$ to its noncommutative counterpart $ E(\cM,\tau)$ (see e.g. \cite{HSZ,DP2,DDS2014,DDST,DPS,DDP2}).
In particular,
$E(\cM,\tau)$ is fully symmetric whenever $E(0,\infty)$ is a fully symmetric function space on $(0,\infty)$ \cite{DP2,DPS}.

\subsection{$\theta$-H\"older functions}
Throughout this section, we always assume that $0<\theta<1$.
A function $f:\bR \rightarrow \bC$ is called a \emph{$\theta$-H\"older function}  if it satisfies the inequality
\begin{align*}
|f(x)-f(y)|\le {\rm const}  |x-y|^\theta, ~x,y\in \bR.
\end{align*}

Let $\Lambda_\theta $ denote the space of all $\theta$-H\"older functions.
Suppose that  $f:\bR\rightarrow \bR$ is a continuous function and $p>0$.
 If $f(A)-f(B)\in S_p$ for any bounded self-adjoint operators $A,B\in B(\cH)$ with  $rank(A-B)<\infty$,
  then
$g:= f\circ h\in B_{p,p}^{1/p} $ for any rational function $h$ that is real on $\bR$ and has no pole at $\infty$ \cite[Theorem 9.2]{AP2},
where $B^{s}_{p,q}$, $0<s,p,q<\infty$ denotes the Besov space on $\bR$ defined by using Littlewood-Paley decomposition  (see \cite[Section 2.6.1]{Triebel2}). 
 In particular,
if $f$ is operator $\theta$-H\"older with respect to   $\left\|\cdot \right\|_p$, then $f\in B_{p, p}^{1/p}$ locally, i.e., the restriction of $f$ to an arbitrary finite interval can be extended to a function of class $B_{p, p}^{1/p}$  (see \cite[Theorem 9.2]{AP2} and the comments below).
However, this necessary condition  is very loose.
The main goal of the present paper  to obtain a simple criterion for a function to be  operator $\theta$-H\"older.

Recently, Sobolev \cite{Sobolev,Sobolev17} introduced   the space $S_{d,\theta}$ consists of  functions $f\in C^{d}(\bR\backslash {0})\cap C(\bR)$  such that
  \begin{align}\label{def:funct}
  \left\|f\right\|_{S_{d,\theta}} :=
  \max_{0\le k \le d }
   \left\|f\right\|_{-\theta+k,k}  :=  \max_{0\le k \le d } \sup _{x\ne 0}
   |x|^{-\theta +k} |f^{(k)}(x)|<\infty .
  \end{align}

The following fact is well-known (see e.g. \cite[(2.3)]{Sobolev17}).
For the sake of completeness, we provide a short proof.
\begin{lemma}\label{Lemmaholdertrivial} For $d\geq 1,$ we have $S_{d,\theta}\subset\Lambda_{\theta}.$
\end{lemma}
\begin{proof}
When $xy\le 0$, it is clear that
$$|f(x)-f(y)| \le |f(x)| +|f(y)|\le \left\|f\right\|_{S_{0,\theta}}|x|^\theta+\|f\|_{S_{0,\theta}}|y|^\theta \le 2\left\|f\right\|_{S_{0,\theta}}|x-y|^\theta.$$

When $xy\geq0,$ we have
$$|f(x)-f(y)|=\Big|\int_x^yf'(t)dt\Big|\leq \left\|f\right\|_{S_{1,\theta}}\Big|\int_x^y|t|^{\theta-1}dt\Big|=\frac1{\theta}\left\|f\right\|_{S_{1,\theta}}\Big||x|^{\theta}-|y|^{\theta}\Big|.$$
Note that
$$\Big||x|^{\theta}-|y|^{\theta}\Big| \stackrel{\tiny \mbox{\cite[(2.2)]{KPR}}}{\le} \Big| |x| -|y | \Big|^\theta  = |x-y|^{\theta},\quad x,y\in\mathbb{R}.$$
Hence, for $xy\geq0,$ we have
$$|f(x)-f(y)|\leq \frac1{\theta}\left\|f\right\|_{S_{1,\theta}}|x-y|^{\theta}.$$

In either case, we have
$$|f(x)-f(y)|\leq \frac2{\theta}\left\|f\right\|_{S_{1,\theta}}|x-y|^{\theta}.$$
\end{proof}

\begin{remark}
The definition of $S_{d,\theta}$ can be modified to the case of $f\in C^{d}(\bR\backslash {x_0})\cap C(\bR)$, $x_0\in \bR$, and the results of the present paper  remain true (see \cite[Remark 2.6]{Sobolev}).
\end{remark}

We note that   $\left\|\cdot \right\|_{-\theta+k,k}$ is an analogue of the
 Schwartz semi-norm  with multi-indices $(k-\theta, k)$, $0\le k\le d$.
The space $S_{d,\theta}$ embraces a wide class of functions, e.g.,
if $f$ is a Schwartz class function with $f(0)=0$, then   $f\in S_{\infty,\theta}$.
It is  easy to see that
 classic operator-monotone functions such as $t\mapsto |t|^\theta$,
 $t\mapsto \log(|t|+1) $ and $t\mapsto \frac{|t|}{r+|t|} $, $t\in \bR$ are in $ S_{\infty,\theta}$.

In this paper, we mainly consider $\theta$-H\"older functions of the class $S_{d,\theta}$.
The following proposition demonstrates that $f\in S_{d,\theta}$ is a very mild condition  for a  $\theta$-H\"older function $f$.
\begin{proposition}
Let $f:\bR \rightarrow \bR \in C^n(\bR \setminus \{0\}) \cap C(\bR) $, $n\ge 0$, with $f(0)=0$ and $\theta\in (0,1)$.
Assume that there exist positive numbers $r,C,\varepsilon>0$  such that
$$
\frac{|f^{(n)}(x)|}{|x|^{\theta-r-n}} \ge C>0
$$
for all $x\in (0,\varepsilon)$.
Then, $f$ is not $\theta$-H\"older.
\end{proposition}
\begin{proof}
Without loss of generality, we assume that $r\in (0,\theta)$.
We assert that if   for $k\ge 2 $ and a constant  $C$, there exists a constant $C_{k,\theta,r} >0$ such that
\begin{align}\label{fnrn}
 \left|f^{(k)}(x)- C \right | \ge  C_{k,\theta,r} x^{\theta -r-k} ,~\forall x\in (0,\varepsilon),
 \end{align}
 then there exist  constants  $C_{k-1,\theta,r} >0$ and  $C'$ such that
\begin{align}\label{fnrn'}
 \left|f^{(k-1)}(x)-C' \right | \ge  C_{k-1,\theta,r} x^{\theta -r -k+1} ,~\forall x\in (0,\varepsilon).
 \end{align}
Without loss of generality, we may assume that
$$ f^{(k)}(x)- C \ge  C_{k,\theta,r} x^{\theta -r-k } ,~\forall x\in (0,\varepsilon). $$
Integrating  from $x$ to $\varepsilon$, we obtain that
$$ f^{(k-1)}(\varepsilon)  - f^{(k-1)}(x ) -  C(\varepsilon-x )  \ge \frac{ C_{k,\theta,r}}{\theta -r-k+1 } \left(\varepsilon ^{\theta -r-k+1} - x^{\theta -r-k+1}  \right),~x\in (0,\varepsilon). $$
Since any polynomial is bounded on $(0,\varepsilon)$, it follows that there exists a constant $C'$ such  that
$$ f^{(k-1)}(x ) - C' \le \frac{ C_{k,\theta,r}}{\theta -r-k+1 } x^{\theta -r-k+1}  ,~x\in (0,\varepsilon). $$
Noting that  $\theta-r-k+1<0$, we obtain  the validity of \eqref{fnrn'}.

Assume that there exist constants $C_{1,\theta,r}>0$ and  $C$  such that
$$|f'(x) -C |\ge C_{1,\theta,r} x^{\theta -r -1}, ~\forall x\in (0,\varepsilon).$$
Without loss of generality, we may assume that $f'(x) -C \ge C_{1,\theta,r} x^{\theta -r -1}, ~\forall x\in (0,\varepsilon)$.
Hence,
$$f(x) - Cx  =f(x) -f(0) - Cx \ge \frac{C_{1,\theta,r} }{\theta -r}x^{\theta -r }, ~\forall x\in (0,\varepsilon).$$
Therefore,
\begin{align*} \sup_{x\ne 0} \left| \frac{f(x)}{x^\theta} \right|   &\ge  \sup_{x\ne 0} \left| \frac{f(x)}{x^\theta}  -Cx^{1-\theta}  \right| - \left| Cx^{1-\theta}  \right| \\
 &\ge  \sup_{x\in (0,\varepsilon)}  \left| \frac{C_{1,\theta,r} }{\theta -r}x^{ -r } \right| - \left| Cx^{1-\theta}  \right|  =\infty .
\end{align*}
This together with \eqref{fnrn'} implies that $f$ is not $\theta$-H\"older.
\end{proof}
\section{Double operator integrals}
In this section, we review some aspects of the beautiful theory of double operator integrals.
Double operator integrals appeared in the paper by Daletskii and Krein \cite{DK}.
For details of the theory of double operator integrals, the reader is referred to \cite{PS1,BS03,AP2016,ST}.
For basic properties of noncommutative $L_p$-spaces $L_p(\cM,\tau)$, we refer to \cite{PX}.

Symbolically, a double operator integral is defined by the formula
$$T_a^{A,B}(V) = \int_{\bR^2} a(\lambda,\mu) dE_A(\lambda)VE_B(\mu),\quad V\in L_2(\cM,\tau)$$
for a bounded Borel function $a$ on $\bR^2$ and for self-adjoint operators $A$ and $B$ affiliated with $\cM$ (see e.g. \cite{dWS}, see also \cite{BS03}\cite[Section 3.5]{ST}). Here, $(\sX,E_A)$ and $(\sY,E_B)$ are spectral measures of $A$ and $B,$ respectively.


We write $a\in \fM_p$ if
$$\left\| a \right\|_{\fM_p}:= \sup_{\cM}\sup_{A=A^*,B=B^* \eta \cM }\left\|T_a^{A,B}\right\|_{L_p \to L_p} <\infty,$$
where $\left\|T_a^{A,B}\right\|_{L_p \to L_p}$ is the operator quasi-norm of $T_a^{A,B}$ from $(L_p\cap L_2)(\cM,\tau)$ into itself.
It is clear that   for every $p>0$, we have
\begin{align}\label{beq1}
\|ab\|_{\fM_p} \le \|a\|_{\fM_p} \|b\|_{\fM_p}
\end{align}
 for every $a,b\in \fM_p$.
It follows from the $p$-triangle inequality for $L_p(\cM,\tau)$, $0<p\le 1$, and every $a,b\in \fM_p$,  we have (see e.g. \cite[(1.2)]{AP1})
\begin{align}\label{beq2}
\|a+b\|_{\fM_p}^p\leq\|a\|_{\fM_p}^p+\|b\|_{\fM_p}^p.
\end{align}
Moreover,  if $a_n \in \fM_p$, $n\ge 0$, such that  $\sum_{n=0}^\infty a_n$ is bounded Borel function,  then 
\begin{align}\label{beq3}
 \left\|\sum_{n=0}^\infty a_n\right\|_{\fM_p}^p\leq\sum_{n=0}^{\infty}\left\|a_n\right\|_{\fM_p}^p.
 \end{align}

Let $ f $ be a function admitting a representation
\begin{align*} 
 f(x,y) = \sum_{n\ge 0} \varphi_n(x) \psi_n (y),
 \end{align*}
where $\varphi_n \in L^\infty (E_A), \psi_n\in L^\infty(E_B )$ are Borel functions, and
\begin{align*}
\sum_{n\ge 0} \left\|\varphi_n \right\|_\infty \left\|\psi_n \right\|_\infty <\infty.
\end{align*}
Then, for arbitrary $V\in B(\cH)$,
\begin{align*}
T_{f}^{A,B}(V) := \sum_{n\ge 0} \left( \int _\sX \varphi_n (\lambda ) d E_A (\lambda)  \right)V \left( \int _\sY \psi_n (\mu) d E_B(\mu)   \right).
\end{align*}
Here, the convergence is understood in the uniform norm topology, and therefore, in (local) measure topology.
In particular, $T_{f}^{A,B}$ is a bounded operator from $\cM$ into $\cM$ \cite{Peller1985} (see also \cite[Theorem 4]{PS2009}).

The following lemma provides a criterion for verifying whether a function is in $\fM_p$ or not (see also \cite[p. 279]{AP1} and \cite[Lemma 2.3]{Ricard}).
\begin{lemma}\label{remark:p}
Let $p\in (0,1]$ and let $a$ be a Borel function admitting a representation
\begin{align}\label{rep}
a(\lambda,\mu)=\sum_{n\ge 0}\varphi_n(\lambda)\psi_n(\mu), ~\lambda,\mu\in \bR,
\end{align}
where $(\left\|\varphi_n\right\|_\infty )_n\in \ell_\infty $ and  $(\left\|\psi_n\right\|_\infty )_n \in \ell_p $.
We have
\begin{equation}\label{mfp estimate}
\left\|a\right\|_{\fM_p}\leq\Big\|\Big\{\|\varphi_n\|_\infty\Big\}_{n\geq0}\Big\|_{\infty}\Big\|\Big\{\|\psi_n\|_\infty\Big\}_{n\geq0}\Big\|_p.
\end{equation}

\end{lemma}
\begin{proof}  Let $a$ be a Borel function admitting  representation \eqref{rep} and such that the right hand side of \eqref{mfp estimate} is finite.
Let   $A$ and $B$ two self-adjoint operators affiliated with $\cM$.
 Since $\ell_p\subset \ell_1,$ it follows that
$$\sum_{n\geq0}\|\varphi_n\|_\infty\|\psi_n\|_\infty<\infty.$$
We have
$$T^{A,B}_a(V)=\sum_{n\geq0}\varphi_n(A)V\psi_n(B),\quad V\in \cM,$$
where convergence is in the norm topology of $\cM$.
 If $V\in L_p(\cM,\tau)\cap \cM,$ then 
\begin{align*}\left\|T^{A,B}_a(V)\right\|_p^p\leq\sum_{n\geq0} \left\|\varphi_n(A)V\psi_n(B) \right\|_p^p
&\leq
\Big(\sum_{n\geq0}\|\varphi_n(A)\|_{\infty}^p\|\psi_n(B)\|_{\infty}^p\Big)\cdot\|V\|_p^p\\
&\leq\Big(\sum_{n\geq0}\|\varphi_n\|_{\infty}^p\|\psi_n\|_{\infty}^p\Big)\cdot\|V\|_p^p.
\end{align*}
Therefore,
$$\left\|T^{A,B}_a\right\|_{L_p\to L_p}\leq\Big(\sum_{n\geq0}\|\varphi_n\|_{\infty}^p\|\psi_n\|_{\infty}^p\Big)^{\frac1p}\leq\Big\|\Big\{\|\varphi_n\|_\infty\Big\}_{n\geq0}\Big\|_{\infty}\Big\|\Big\{\|\psi_n\|_\infty\Big\}_{n\geq0}\Big\|_p.$$
Taking the supremum over $A=A^*,B=B^*\eta\cM$ and then over $\cM,$ we complete the proof.
\end{proof}

The following is an easy corollary of Lemma \ref{remark:p}. Similar results are obtained in \cite[Corollary 2.8]{Ricard}.
\begin{corollary}\label{alpha beta lemma} Let $0< p\le 1$.   Let
$$\alpha(s,t):=\frac1{s-t}\chi_{[\frac12,1)}(|s|)\chi_{(0,\frac14)}(|t|),~s,t\in\bR$$ and
$$\beta(s,t):=\frac{t}{t-s}\chi_{[\frac12,1)}(|s|)\chi_{(2,\infty)}(|t|)~s,t\in\bR.$$
Then, we have
 $$\left\| \alpha \right\|_{\fM _p}, \left\| \beta \right\|_{\fM _p} < \infty.$$
\end{corollary}
\begin{proof}
Note that
$$\frac{1}{1-x} = \sum_{n\ge 0} x^n, ~|x|<1.$$

We write
$$\alpha(s,t)=\frac 1 s \frac{1}{1-\frac {2t}{2 s}}\chi_{[\frac12,1)}(|s|)\chi_{(0,\frac14)}(|t|) = \frac{1}{s}  \sum_{n\geq0}(2s)^{-n}\chi_{[\frac12,1)}(|s|) (2t)^{n}\chi_{(0,\frac14)}(|t|) $$
for any $s,t\in\bR$.
Since $\sup_{s} (2s)^{-n}\chi_{[\frac12,1)}(|s|)<1$,
it follows from Lemma \ref{remark:p} that
$$\left\|\alpha\right\|_{\fM_p}\leq  2
\left\|\{ (2s)^{-n}\chi_{[\frac12,1)}(|s|)   \}_{n\geq0}\right\|_{\infty}
\left\|\{(1/2)^{n}\}_{n\geq0} \right\|_p<\infty.$$
This completes the proof that $\left\| \alpha \right\|_{\fM_p} <\infty$.

For $\beta$, we
 write
$$\beta(s,t)=\frac{1}{1-\frac{s}{t}}\chi_{[\frac12,1)}(|s|)\chi_{(2,\infty)}(|t|) =\sum_{n\geq0}s^n\chi_{[\frac12,1)}(|s|)\cdot t^{-n}\chi_{(2,\infty)}(|t|)$$
for any $s,t\in\bR$.
By Lemma \ref{remark:p}, we have
$$\left\|\beta\right\|_{\fM_p}
\leq
\left\|\{2^{-n}\}_{n\geq0}\right\|_{\infty}
\left\|\{2^{-n}\}_{n\geq0}  \right\|_p<\infty.$$
\end{proof}
\begin{remark}\label{prop:1dbounded} For a bounded Borel function $g:\mathbb{R}\to\mathbb{C},$ define a bounded Borel function $f:\bR^2 \rightarrow \bC$ by setting $f(x,y)=g(x).$ We have
$$\|f\|_{\fM_p} = \|g\|_{\infty}.$$
\end{remark}

If  $f$ is a Lipschitz function on $\bR,$ then
\begin{align}\label{eq:dif}
f(A +K )-f(A)=T_{\fD f }^{A+K,A}(K)
\end{align}
for self-adjoint operators $A \in S(\cM,\tau)$ and $K \in L_2(\cM,\tau) $ \cite[Theorem 7.4]{dWS}.  Here, we use the notation
$$(\fD f)(x,y) := \frac{f(x) -f(y)}{x-y}, ~x\ne y, \qquad  (\fD f)(x,x) :=f'(x),\quad x,y\in \bR.$$
Note  that for a  differentiable function $f$ and for $x\neq y,$ we have
\begin{align}\label{difff}
(\fD f)(x,y)=\int_0^1f'(tx+(1-t)y)dt.
\end{align}

 Let $k\in \mathbb{N}$, $1\le p \le \infty$.
The Sobolev space $W^{k,p}(\bT^2)$ is defined to be the set of all functions $f$ on $\bT^2$ such that for every multi-index $\alpha=(\alpha_1,\alpha_2)$ with $|\alpha|\le k$, the mixed partial derivative $\frac{\partial^{|\alpha|} f}{\partial _2^ {\alpha_2} \partial_1 ^{\alpha_1} }$
exists in the weak sense and is in $L_p(\Omega)$\cite{Triebel2}.
Here,
 $\bT^2$ stands for the 2-dimensional torus, i.e., $\bT=\bR / 2\pi \bZ$ (equipped with the Lebesgue measure).
The following lemma is a modification of \cite[Lemma 2.7]{Ricard}.
Note that  the function considered below is not required to be infinitely differentiable for the second argument as in \cite[Lemma 2.7]{Ricard}.

\begin{lemma}\label{multiplier fourier lemma}  Let $b\in \mathbb{Z}$ be such that  $b>1/p$, $p\in (0,1]$.
Let $a:\mathbb{R}^2\to\mathbb{C}$ be a bounded Borel $2\pi$-periodic function in the both arguments.
If the mixed partial derivative $\frac{\partial^{m+n}a}{\partial_2^ m \partial_1   ^n }   $ exists (in the usual sense)
for every multi-index $(m,n)$, $0\le m+n\le b+1$
 with
$a\in W^{b+1,2}(\bT^2)$ (that is, $\left\|a\right\|_{W^{b+1,2}(\bT^2) } : =\sum _{0\le m+n \le b+1}\left\|\frac{\partial^{m+n} }{(\partial _2)^m (\partial_1)^{n}}a\right\|_2<\infty$, where   $\|f\|_2$ denotes the $L_2$-norm of $f |_{\bT^2  }$), then we have
$$\left\|a\right\|_{\fM_p}\leq c_{p,b} \left\|a\right\|_{W^{b+1,2}(\bT^2) }  <\infty .$$

\end{lemma}
\begin{proof} We use Fourier representation in the second argument
$$a(x,y)=\sum_{n\in\mathbb{Z}}a_n(x)e_n(y)=a_0(x)+\sum_{n\neq0}|n|^ba_n(x)\cdot |n|^{-b}e_n(y),$$
where $e_n (y)= e^{iny}$, $y\in \bR$ and $a_n(x) = \frac{1}{2\pi}  \int_{-\pi}^{\pi} a(x,y)e^{-iny} dy$, $x\in \bR$, $n\in \bZ$.
By Lemma \ref{remark:p}, we have
\begin{align}\label{ineq:fora}
\|a\|_{\fM_p}&\leq\|a_0\|_{\infty}+\Big\|\Big\{\||n|^ba_n\|_\infty\Big\}_{n\neq0}\Big\|_{\infty}\Big\|\Big\{\||n|^{-b}e_n\|_\infty\Big\}_{n\neq0}\Big\|_p\nonumber\\
&=\|a_0\|_{\infty}+c_{p,b}\Big\|\Big\{\||n|^ba_n\|_\infty\Big\}_{n\neq0}\Big\|_{\infty},
\end{align}
where $c_{p,b} : = (\sum_{n\ne 0} |n|^{-pb})^{1/p}<\infty $.
For every continuously differentiable function $h$ on $\bT$, integrating by part, we get (see e.g. \cite[Theorem 5.1]{Sugiura})
\begin{align}\label{eq:handh'}
\hat{h}(n) = \frac{1}{2\pi}   \int_{-\pi}^\pi h(t)e^{-int}  d t=\frac{1}{2\pi}  \frac{1}{ni}\int_{-\pi}^\pi h'(t)e^{-int}  d t =\frac{1}{ni} \hat{h'}(n),
\end{align}
where $\hat{h}(n)$ stands for the $n$-th Fourier coefficient of $h$.
By Cauchy's inequality,
we have
\begin{align}\label{eq:bypart}
\|h\|_{\infty}
\leq \sum_{n}|\hat{h}(n)|
&\le |\hat{h}(0)| +  \left(\sum_{n\ne 0} |\hat{h}(n)|^2 \cdot n^2 \right)^{1/2} \left(  \sum_{n\ne 0} n^{-2} \right)^{1/2}  \nonumber \\
&\leq   \|h\|_2+  \left(\sum_{n\ne 0} |\widehat{h'}(n)|^2  \right)^{1/2} \left(\frac{\pi^2}{3}\right)^{1/2}    \nonumber  \\
&\le  \|h\|_2+ \left(\frac{\pi^2}{3} \right)^{1/2} \left\|h'\right\|_2,
\end{align}
Applying \eqref{eq:bypart} to every $a_n$ in \eqref{ineq:fora}, we have
\begin{align*}
\|a\|_{\fM_p}
& \stackrel{\eqref{ineq:fora}}{\le} \|a_0\|_{\infty}+c_{p,b}\Big\|\Big\{\||n|^ba_n\|_\infty\Big\}_{n\neq0}\Big\|_{\infty},\\
&\stackrel{\eqref{eq:bypart}}{\le}   \|a_0\|_2+ \left(\frac{\pi^2}{3} \right)^{1/2} \left\|a_0'\right\|_2     \\
&\qquad \qquad +
c_{p,b}\left\|\left\{ \left\||n|^b\left( \|a_n \|_2+ \left(\frac{\pi^2}{3} \right)^{1/2} \left\|a_n'\right\|_2 \right)\right\|_\infty\right\}_{n\neq0}\right\|_{\infty},\\
&\leq \|a_0\|_2+ \left(\frac{\pi^2}{3}
\right)^{1/2}\|a_0'\|_2\\
&\qquad \qquad +c_{p,b}\Big\|\Big\{|n|^b\|a_n\|_2\Big\}_{n\neq0}\Big\|_{\infty}+ \left(\frac{\pi^2}{3}
\right)^{1/2} c_{p,b}\Big\|\Big\{|n|^b\|a_n'\|_2\Big\}_{n\neq0}\Big\|_{\infty}\\
&\leq \|a_0\|_2+ \left(\frac{\pi^2}{3} \right)^{1/2}\|a_0'\|_2\\
&
\qquad \qquad +c_{p,b}\Big\|\Big\{|n|^b\|a_n\|_2\Big\}_{n\neq0}\Big\|_2+ \left(\frac{\pi^2}{3}\right)^{1/2} c_{p,b}\Big\|\Big\{|n|^b\|a_n'\|_2\Big\}_{n\neq0}\Big\|_2.
\end{align*}
By Pythagorean identity and for every $x\in \bR$,  letting  $a_x(y)=a(x,y)$, $y\in \bR$, we obtain that
$$\Big\|\Big\{|n|^b\|a_n\|_2\Big\}_{n\neq0}\Big\|_2
=\left\| \sum_{n\neq0} |n|^b a_n e_n \right\|_2
 \stackrel{\eqref{eq:handh'}}{=} \left\| \sum_{n\neq0}  |\widehat{a_x^{(b)}}(n )| e_n \right\|_2   \le  \left\|
 \frac{\partial^b}{\partial_2^b}a
 \right\|_2,$$
and, for every $x\in \bR$, letting  $c_x(y)=\frac{\partial}{\partial_1 } a(x,y)$, $y\in \bR$, we have
$$\Big\|\Big\{|n|^b\|a_n'\|_2\Big\}_{n\neq0}\Big\|_2
=\left\| \sum_{n\neq0} |n|^b a'_n e_n \right\|_2
 \stackrel{\eqref{eq:handh'}}{=}
  \left\| \sum_{n\neq0}    |\widehat{c_x^{(b)}}(n)|  e_n \right\|_2
 \le
\left\|\frac{\partial ^{b+1}}{\partial_2^b\partial_1 } a\right\|_2.$$
This completes the proof.
\end{proof}

For $r>0$, the dilation operator $\sigma_r$ are  acting on
any bounded Borel function $a:\bR  \rightarrow \bC$  such that
$$(\sigma_r a ) (s):= a\left(\frac{s}{r}\right),\quad s\in\mathbb{R}.$$
In particular, if $f\in S_{d,\theta}$, then
\begin{align}\label{eqdia}
\left\|\sigma_r f\right\|_{S_{d,\theta}}  = \max_{0\le k \le d } \sup _{x\ne 0}    \frac{|(\sigma_r f)^{(k)}(x)|}{|x|^{\theta  -k}}
&=  \max_{0\le k \le d } \sup _{x\ne 0}    \frac{|r^{-k} f^{(k)}(\frac xr)|}{|x|^{\theta  -k}}
\nonumber
 \\
 &
 =  \max_{0\le k \le d } \sup _{x\ne 0}    \frac{| f^{(k)}(\frac x r )|}{ r^\theta  |\frac x r|^{\theta  -k}}
 =\frac{1}{r^\theta}\left\|
f\right\|_{S_{d,\theta}}.
\end{align}

Similarly, for $r>0$, the dilation operator $\sigma_r$ are  acting on
any bounded Borel function $a:\bR^2 \rightarrow \bC$  such that
$$(\sigma_r a ) (s,t):= a\left(\frac{s}{r},\frac{t}{r} \right),\quad s,t\in\mathbb{R}.$$
\begin{remark}\label{remark:sigma} For any bounded Borel function $a:\bR^2 \rightarrow \bC$ and self-adjoint operators $A,B\eta \cM$,
and  for any $ V\in L_2(\cM,\tau)$,  we have
$$T_a^{A,B}(V) = \int_{\bR^2} a(\lambda,\mu) dE_A(\lambda)VE_B(\mu)= \int_{\bR^2} a\left( \frac \lambda r, \frac \mu r\right) dE_{rA}( \lambda)VE_{rB}(\mu) = T_{\sigma_r a}^{rA,rB}(V).
$$
In particular, we have
$$\|\sigma_ra\|_{\fM_p}=\|a\|_{\fM_p}.$$
\end{remark}

 The following result is an easy consequence of \cite[Theorem 3.1]{PS2008} (by taking $x=pq$).
 Similar results have been proven in \cite[Proposition 6.11]{dS} and \cite[Lemma 7.3]{dWS}.

\begin{lem}\label{lemma 3.6}
Assume that $f:\bR\rightarrow \bC$ is a Lipschitz function.
 Let $A,B\in S(\cM,\tau)$ be self-adjoint operators and let $p,q$ be spectral  projections of  $A$ and $B,$ respectively. If $p(A-B)q\in L_2(\cM,\tau),$ then
$$T^{A,B}_{\fD f}(p(A-B)q)=p(f(A)-f(B))q.$$
\end{lem}

\section{Double operator integrals of divided differences on $L_p(\cM,\tau)$, $0< p\le 1$}\label{DOI}


Let  $0<\theta<1$.
Assume that $\cH$ is a separable Hilbert space.
In this section, we deal with the subspace $S_{d,\theta}$ of the $\theta$-H\"{o}lder class $\Lambda _\theta$.

%

We define $d:=d(p)$, $p>0$,
by setting
that $d$ is  the minimal integer such that
 $$d > \frac1p+2 , ~p\le   1$$
 and $$ d(p )=d(1)=4, ~p>1.$$
 As mentioned before,  when $p=\infty$, the operator $\theta$-H\"older functions has been  described  in \cite{AP3,NF,APPS}.
  Hence, we only consider the case for $0<p<\infty$.
Without loss of generality, we only prove the case for $0<p\le 1$.
Indeed,
the case when $p>1$ follows immediately from
 the case for $p =1 $,
 the Hardy--Littlewood--Polya inequality \cite[Corollary 2.5]{HSZ} (see also \cite[Theorem 11]{DSZ} and \cite[Chapter I, Theorem D.2]{MOA}) and Theorem \ref{th:subm}.

The argument in \cite{Ricard} heavily rests on the homogeneity of the function $t\mapsto |t|^\theta$, $t\in \bR$, which is one of the difficulties we encountered. A different method from that in  \cite{Ricard} is required in order to consider a much more general class of operator $\theta$-H\"older functions.





Throughout this section, 
we always assume that $p\in (0,1]$ and $f\in S_{d,\theta}$.

We note that
a result similar  to the following lemma  was claimed in \cite[Theorem 9.6]{BS1977} under slightly weaker conditions.
We would like to thank Edward McDonald for providing us a proof of \cite[Theorem 9.6]{BS1977}.
\begin{lem}\label{local lemma}  Let $f\in S_{d,\theta}$ and let $\phi$ be a smooth function, If $\phi $ is supported on $(\varepsilon, \pi)$, then
$$\|(\phi\otimes\phi)\cdot\fD f\|_{\fM_p}\leq c_{p,\phi}\|f\|_{S_{d,\theta}}.$$
Here, $(\phi\otimes\phi)\cdot\fD f$ is defined by setting $(\phi\otimes\phi)\cdot\fD f (x,y ) = \phi(x)\phi(y)\fD f(x,y)$, $x,y\in (\varepsilon, \pi)$.
\end{lem}
\begin{proof}  Let $a:\mathbb{R}^2\to\mathbb{C}$ be a Borel function $2\pi$-periodic in both arguments such that
$$a=(\phi\otimes\phi)\cdot\fD f\mbox{ on }[\varepsilon,\pi]^2.$$
By Lemma \ref{multiplier fourier lemma}, for any $b>1/p$,  we have
$$\|(\phi\otimes\phi)\cdot\fD f\|_{\fM_p}\leq c_{p,b}\|a \|_{W^{b+1,2}(\mathbb{T}^2)}.$$
Since $\phi=0$ on $(-\varepsilon,\varepsilon)$, it follows that there exists a constant $c_b$ such that
\begin{align}\label{ineq:Kxy}
\|a \|_{W^{b+1,2}(\mathbb{T}^2)}
&=
  \|(\phi\otimes\phi)\cdot\fD f\|_{W^{b+1,2}([-\pi,\pi]^2)} \nonumber \\
&\leq c_b  \Big(\max_{0\leq k\leq b+1}\left\|\phi^{(k)}\right\|_{\infty}\Big)^2\|\fD f\|_{W^{b+1,\infty}([\varepsilon,\pi]\times [\varepsilon,\pi] )},
\end{align}
where $\|\fD f\|_{W^{b+1,\infty}([\varepsilon,\pi]\times [\varepsilon,\pi] )} = \sum_{0\le   m+n \le b+1} \left\| \frac{\partial^{m+n}}{(\partial_2)^m (\partial_1)^n}   (\fD f ) \chi_{[\varepsilon,\pi]\times [\varepsilon,\pi]} \right\|_\infty $.

Take $b=d-2$.
For $x,y \in [\epsilon,\pi]$, we have $tx + (1-t)y \ge\varepsilon$.
By the Leibniz integral rule, we obtain that for any $m,n\ge 0$ with $m+n \le d-1$,
\begin{align*}
\left|
\frac{ \partial^{n+m} } { \partial y^m \partial x^n}(\fD f)(x,y )
\right|
&\stackrel{\eqref{difff}}{=}
\left|
\int_0^1 \frac{\partial^{n+m} } { \partial y^m \partial x^n }  f '(tx + (1-t)y ) dt
\right|
\nonumber \\
&~\le \left|\int_0^1 t^n (1-t)^m  f^{(1+m+n)}(tx + (1-t)y ) dt \right| \nonumber \\
&~\le  \int_0^1  | f^{(1+m+n)}(tx + (1-t)y )| dt \nonumber\\
&~\stackrel{\eqref{def:funct}}{\le}  \int_0^1  | tx + (1-t)y  | ^{-m-n-1+ \theta }\|f\|_{S_{d,\theta}} dt \nonumber\\
&~\le \varepsilon^{-m-n-1}\|f\|_{S_{d,\theta}},
\end{align*}
which together with \eqref{ineq:Kxy} completes the proof.
\end{proof}
The following lemma contains crucial estimates for our main result in this section, Theorem~\ref{th:3.5}.
\begin{lem}\label{core lemma}
There exists a constant $c_p>0 $ depending on $p$ only such that
\begin{enumerate}[{\rm (i)}]
\item\label{corea}
$$\left\|
(\fD f)\cdot\chi_{[\frac12,1)\times(0,\frac14)}
\right\|_{\fM_p}\leq c_p \left\| f  \right \|_{S_{0,\theta}}.$$
\item\label{coreb}
$$\left\|
(\fD f)\cdot\chi_{[\frac12,1)\times(2,\infty )}
\right\|_{\fM_p}\leq c_p  \left\| f \right \|_{S_{0,\theta}}.$$
\item\label{corec}
$$\left\|
(\fD f)\cdot\chi_{[\frac12,1)\times[\frac14,2]}
\right\|_{\fM_p} \leq  c_p  \left\| f \right \|_{S_{d,\theta}} .$$
\end{enumerate}
\end{lem}
\begin{proof}  \eqref{corea}. Let $\alpha$ be defined as in Corollary \ref{alpha beta lemma} and let $a_0$ and $a_1$ be such that
$$a_0(s,t):=f(s)\chi_{(0,1)}(s),\quad a_1(s,t):=f(t)\chi_{(0,1)}(t),\quad s,t\in\mathbb{R}.$$
It is obvious that
$$(\fD f)\cdot\chi_{[\frac12,1)\times(0,\frac14)}=a_0\alpha-a_1\alpha.$$
By \eqref{beq1} and \eqref{beq2}, we have
$$
\left\|(\fD) f\cdot\chi_{[\frac12,1)\times(0,\frac14)}\right\|_{\fM_p}^p
\leq
\left\| a_0 \right \|_{\fM_p}^p
\left\| \alpha\right\|_{\fM_p}^p
+
\left\| a_1\right\|_{\fM_p}^p
\left\| \alpha \right\|_{\fM_p}^p.$$
By Remark \ref{prop:1dbounded}, we have
$$\|a_0\|_{\fM_p}\leq\|f\chi_{(0,1)}\|_{\infty}\leq\|f\|_{S_{0,\theta}},\quad \|a_1\|_{\fM_p}\leq\|f\chi_{(0,1)}\|_{\infty}\leq\|f\|_{S_{0,\theta}}.$$
The assertion follows now from Corollary \ref{alpha beta lemma}.

\eqref{coreb}. Let $\beta$ be defined as in Corollary \ref{alpha beta lemma} and let $b_0$ and $b_1$ be such that
$$b_0(t,s)=\frac{f(t)}{t}\chi_{(1,\infty)}(t),\quad b_1(t,s)=t^{-1}\chi_{(1,\infty)}(t),\quad s,t\in\mathbb{R}.$$
It is obvious that
$$(\fD f)\cdot\chi_{[\frac12,1)\times(2,\infty )}=b_0\beta-a_0b_1\beta.$$
By \eqref{beq1} and \eqref{beq2}, we have
$$
\left\|
(\fD f)\cdot\chi_{[\frac12,1)\times(2,\infty )}\right\|_{\fM_p}^p
\leq
\left\| b_0\right\|_{\fM_p}^p
\left\| \beta\right\|_{\fM_p}^p+\left\| a_0\right\|_{\fM_p}^p \left\| b_1\right \|_{\fM_p}^p\left\| \beta \right\|_{\fM_p}^p.$$
By Remark \ref{prop:1dbounded}, we have
$$\left\| b_0\right\|_{\fM_p}\leq \sup_{t>1}\frac{|f(t)|}{t}\leq \left\| f\right\|_{S_{0,\theta}},$$
$$\left\| a_0\right\|_{\fM_p}\leq \left\| f\chi_{(0,1)} \right\|_{\infty}\leq \left\|f \right\|_{S_{0,\theta}},\quad \left\|b_1\right\|_{\fM_p}\leq\sup_{t>1}\frac1t=1.$$
The assertion follows now from Corollary \ref{alpha beta lemma}.

\eqref{corec} Let $\phi:[0,2\pi]\to[0,1]$ be a smooth function supported in $[\frac{1}{8},\pi]$ which is identically $1$ on $\left[\frac{1}{4},2\right].$ We have
$$(\fD  f) \cdot\chi_{[\frac12,1)\times[\frac14,2]}=\Big((\phi\otimes\phi)\fD f\Big)\cdot\chi_{[\frac12,1)\times[\frac14,2]}.$$
Therefore,
$$\left\|(\fD f )\cdot\chi_{[\frac12,1)\times[\frac14,2]} \right\|_{\fM_p}\leq \left\|(\phi\otimes\phi)\fD f\right\|_{\fM_p}\left\|\chi_{[\frac12,1)\times[\frac14,2]} \right\|_{\fM_p}.$$
 Note that $d$ depends on $p,$ while  $\phi$ is independent of $f,$ $\theta$ and $p$.
 The assertion follows now from Lemma \ref{local lemma}.
\end{proof}

\begin{proposition}\label{4.4}
There exists a constant $c_p$ depending on $p$ only such that
$$\left\|(\fD f) \cdot\chi_{[\frac12,1)\times(0,\infty)}\right\|_{\fM_p} \leq  c_p\|f\|_{S_{d,\theta}}.$$
$$\left\|(\fD  f) \cdot\chi_{(-1,-\frac12]\times(-\infty,0)}\right\|_{\fM_p} \leq  c_p\|f\|_{S_{d,\theta}}.$$
\end{proposition}
\begin{proof}  The first inequality immediately follows from Lemma \ref{core lemma} \eqref{corea}, \eqref{coreb} and \eqref{corec} and \eqref{beq2}. The same argument in Lemma \ref{core lemma} yields the validity of the second inequality.
\end{proof}

The following is a consequence of Remark
 \ref{remark:sigma}, which is the first main result of this section.

\begin{theorem}\label{th:3.5}
There exists a constant $C_{p}$ such that
$$\left\| (\fD f)  \cdot \chi_{[\frac{1}{2^{k+1}},\frac{1}{2^{k}})\times (0,\infty )}  \right\|_{\fM_p } \le
C_{p} \left\| f\right\|_{S_{d,\theta}} 2^{-k(\theta-1) }, ~k\in \bZ  .$$
\end{theorem}
\begin{proof}
We define  the following functions:
$$g_k(s,t):=(\fD f)  \cdot (s,t)\chi_{[\frac{1}{2^{k+1}},\frac{1}{2^{k}})}(s)\chi_{(0,\infty )}(t),\quad 
k\in \bZ, ~s,t\in \bR. $$

By Proposition \ref{4.4}, we have
\begin{align}\label{latter}
\|g_0\|_{\fM_p } \le C_{p} \|f\|_{S_{d,\theta}} .
\end{align}
Applying  the latter result  to  function $\sigma_{2^k }f $ 
and noting that
$$g_k =2^k \sigma_{2^{-k}}\left((\fD \sigma_{2^k} f )\cdot \chi_{[\frac 12 ,1 )\times [0,\infty)}\right)  ,$$
by Remark \ref{remark:sigma}, we have
$$\left\| g_k \right\|_{\frak{M}_p} = 2^k \left\|(\fD \sigma_{2^k} f )\chi_{[\frac 12 ,1 )\times [0,\infty)}\right\|_{\fM_p} \stackrel{\eqref{latter}}{\le}  C_{p}  2^k \|\sigma_{2^k}f\|_{S_{d,\theta}} \stackrel{\eqref{eqdia}}{\le} C_{p} \frac{2^{k}}{2^{k\theta}}\|f\|_{S_{d,\theta}} . $$
%
\end{proof}

The indicator function featuring in Proposition \ref{4.4} and Theorem \ref{th:3.5} are related to measurable sets positioned in the first and third quadrants, which are needed in the next section.
The following lemmas allow us to consider the case of  the indicator functions of sets in  the 2nd or 4th quadrant.
\begin{lemma}
\label{eq:b_000}
Let $a >0$.
Let $$b_0(s,t) := \frac{|s|^\theta}{s-t}\chi_{(-\infty,-a) \times (0,\infty)},~s,t\in \bR,$$
and $$b_1(s,t) := \frac{|t|^\theta}{s-t}\chi_{(-\infty,-a) \times (0,\infty)},~s,t\in \bR.$$
There exists a constant $C_{p,\theta}$ such that
$$ \|b_0\|_{\fM_p} , \|b_1 \|_{\fM_p} \le C_{p,\theta } a^{\theta -1}.$$
\end{lemma}
\begin{proof}
We only prove the assertion for $b_0$, since the case for $b_1$ is very similar.
To show that $\left\| b_0\right\|_{\fM_p}<\infty$, it is equivalent to show that the function $b$ defined by
$$b  (s,t) = \frac{|s|^\theta}{s+ t}\chi_{(a,\infty ) \times (0,\infty)},~s,t\in \bR,$$
satisfies that
$\left\|b\right\|_{\fM_p}<\infty$.

Define
$$I_k : = [2^k a , 2^{k+1} a ], ~k\ge 0 , \mbox{ and }
I_{-1}=[0,a ]
.$$
Assume that $k\ge l \ge -1 $.
By Remark \ref{remark:sigma},  for any $k\ge 0 $ and $l\ne -1$,   we have
\begin{align*}
\left\|\left(\frac{1}{s+t}\right)_{s\in I_k,t\in I_l}\right\|_{\fM_p } & =
\left\| \left(\frac{2^{-k  } \frac 1 a }{s+t}\right)_{s\in [1,2],t\in [2^{l-k} ,2^{l-k+ 1}  ]}\right\|_{\fM_p}\\
&\le
 \frac 1 {2^{ k  }a}
\left\| \left(\frac{1 }{s+t}\right)_{s\in [1,2],t\in [0, 2]}\right\|_{\fM_p};
\end{align*}
 when $l=-1$, we have
 \begin{align*}
 \left\|\left(\frac{1}{s+t}\right)_{s\in I_k,t\in I_l}\right\|_{\fM_p }  &=
\left\| \left(\frac{2^{-k  } \frac 1 a }{s+t}\right)_{s\in [1,2],t\in [0,2^{l-k+ 1}  ]}\right\|_{\fM_p}\\
&\le
 \frac 1 {2^{ k  }a}
\left\| \left(\frac{1 }{s+t}\right)_{s\in [1,2],t\in [0, 2]}\right\|_{\fM_p}.
 \end{align*}
Consider  smooth  functions $\phi_1$ supported  in $[\frac34, \frac94]$ with $\phi_1(t)=1$ on $[1,2]$
and $\phi_2$ supported  in $[-\frac14, \frac94]$ with $\phi_2(t)=1$ on $[0,2]$.
 We define a function $c$   by  $c (s,t):= \phi_1(s)\phi_2(t)\left(\frac{1 }{s+t}\right)$, which is a smooth function supported in $[\frac34, \frac94] \times [-\frac14, \frac94] $.
By Lemma \ref{multiplier fourier lemma}, we have
 $$\|c  \|_{\fM_p} \le c_{ p}   ,$$
where $c_{p}$ depends on   $p$, $\phi_1$ and $\phi_2$ only.
Hence,
$$\left\| \left(\frac{1}{s+t}\right)_{s\in I_k,t\in I_l}\right\|_{\fM_p} \le \frac 1 {2^{ k  }a}  c_{p}. $$
The case when $l\ge k \ge 0$  can be obtained by a similar argument.
Therefore,
$$\left\| \left(\frac{1}{s+t}\right)_{s\in I_k,t\in I_l}\right\|_{\fM_p} \le \frac 1 {2^{ \max\{k,l\}  }a}  c_{p}, ~k \ge 0,~ l\ge -1. $$

On the other hand,
setting $c_1(s)=|s|^\theta$, by Remark \ref{prop:1dbounded},
we have
$$\left\|  c_1 |_{s\in I_k, t\in I_l}  \right\|_{\fM_p } \le \sup\{|s|^\theta :s \in I_k \} =  2^{(k+1)\theta} a^\theta \le 2\cdot 2^{\max\{k,l \}\theta} a^\theta .$$
 Thus,  we obtain that there exists a constant $C_{p}$ such that
\begin{align*}
\left\|b \right\|_{\fM_p }^p
&\stackrel{\eqref{beq3}}{\leq} a^{p(\theta-1)} \sum_{k\ge 0,l\geq -1}C_{p} 2^{p  \theta  \max\{k,l\}}2^{-p \max\{k,l\}}\\
&\leq  a^{p(\theta-1)} \sum_{k\ge 0,l\geq -1}C_{p} 2^{-p(1-\theta)\max\{k,l\}}\\
&\le    a^{p(\theta-1)} \sum_{k \geq 0 } (k+2) C_{p}2^{-p(1-\theta)k } <\infty,
\end{align*}
which completes the proof.
\end{proof}

\begin{lemma}\label{corec2}
Given $a>0$, we have
$$ \left\|  (\fD f)  \cdot \chi_{(-\infty, -a] \times [0,\infty ) }   \right\|_{\fM_p}\le C_{p,\theta} \left\|f \right\|_{S_{d,\theta}} a^{\theta -1 }. $$

\end{lemma}
\begin{proof}
Let $b_0$ and $b_1$ be defined as in Lemma \ref{eq:b_000}.
Let $$a_0(s,t) :=\frac{f(s)}{|s|^\theta},~s \ne 0, \mbox{ and } a_1(s,t) :=\frac{f(t)}{|t|^\theta},~t \ne 0. $$
Then, we have
$$(\fD f)   \cdot \chi_{(-\infty, -a] \times [0,\infty ) } = a_0 b_0 - a_1b_1. $$
By Remark \ref{prop:1dbounded},
$$\|a_0\|_{\fM_p}, \|a_1\|_{\fM_p} \le \|f\|_{S_{0,\theta}}.$$
The assertion follows from the $p$-th triangular inequality \eqref{beq2} and Lemma \ref{eq:b_000}.
\end{proof}

\section{Operator $\theta$-H\"{o}lder functions with respect to $\left\|\cdot\right\|_p$, $p>0$}\label{bounded}


In this section,
without loss of generality (see Remark \ref{non}),  we may assume that $\cM$ is a semifinite von Neumann algebra acting on    a separable Hilbert space $\cH$.
We study the operator $\theta$-H\"{o}lder functions  with respect to $\left\|\cdot\right\|_p$, $p>0$.
Again, since the case when $p=\infty$ has been thoroughly treated in \cite{AP3,APPS,NF},
we only consider the case when $0<p<\infty$.
Moreover, the assertion for  $p>1$ is a  consequence of
 that  for $p =1 $ together with
 the Hardy--Littlewood--Polya inequality \cite[Corollary 2.5]{HSZ} and Theorem \ref{th:subm}.
 Therefore, unless otherwise  stated,  we always assume that $$0< p\le 1 .$$

For the sake of convenience, we denote $s(X)_+ := E^{X}(0,\infty)$ and $s(X)_- : = E^{X}(-\infty,0)$, and $n(X) ={\bf 1} - s(X)$ for $X\in S_h(\cM,\tau)$ (see Section \ref{s:p}).
Assume that $f\in S_{d(p),\theta}$.
Note that \begin{align}\label{eq:decomDF}
f(A)-f(B) &=  s(A)_+\cdot (f(A)-f(B))\cdot s(B)_+ + s(A)_+\cdot (f(A)-f(B))\cdot s(B)_- \nonumber\\
&\quad + s(A)_-\cdot (f(A)-f(B))\cdot s(B)_+  + s(A)_-\cdot (f(A)-f(B))\cdot s(B)_-\nonumber\\
& \quad + n(A)\cdot (f(A)-f(B))\cdot s(B)  +  s(A)\cdot (f(A)-f(B))\cdot n(B).
\end{align}
The main result of this section is Theorem \ref{th:main}, which provides the estimates of the quasi-norm $\left\|f(A)-f(B)\right\|_p$ in terms of $\left\| |A-B|^\theta \right\|_p$.
We shall provide a detailed proof of the estimates of the quasi-norms of $s(A)_+\cdot (f(A)-f(B))\cdot s(B)_+$ and $s(A)\cdot (f(A)-f(B))\cdot n(B) $.
The proof for the  other cases is exactly the same, where the proof in
the case of $s(A)_+\cdot (f(A)-f(B))\cdot s(B)_-$ and $s(A)_-\cdot (f(A)-f(B))\cdot s(B)_+ $
 require Lemma \ref{corec2} instead of Theorem \ref{th:3.5}.

For the sake of convenience,  we use the following notations.
\begin{notation} \label{nota}
Set  $I_k:=[2^{-k-1}, 2^{-k})$ and $J_k:= (0,2^{-k})$, $k\in \bZ$.
Assume that $A$ and $B$ are in $S(\cM,\tau)_h$.
We set
\begin{align*}
g_k(s,t):=(\fD f) (s,t)\chi_{[2^{-k-1}, 2^{-k}) }(s) \chi_{(0,\infty)}(t) ,&\quad h_k(s,t):=(\fD f) (s,t ) \chi_{[2^{-k-1}, 2^{-k}) }(t) \chi_{(0,\infty)}(s),\\
p_k:=\chi_{[2^{-k-1}, 2^{-k}) }(A),&\quad q_k:=\chi_{[2^{-k-1}, 2^{-k})}(B),\\
P_k:=\chi_{(0,2^{-k})}(A),&\quad Q_k:=\chi_{(0,2^{-k}) }(B),\\
V_k:=p_k(A-B)Q_k,&\quad W_k:=P_{k+1}(A-B)q_k.\end{align*}
\end{notation}

\begin{lem}  Let $A,B\in S(\cM,\tau)$ be such that $A-B\in L_p(\cM,\tau)$.
We have
\begin{align}\label{Remark:5.1}p_k(f(A)-f(B))Q_k=T^{A,B}_{g_k}(V_k),\quad P_{k+1}(f(A)-f(B))q_k=T^{A,B}_{h_k}(W_k)
.
\end{align}

\end{lem}
\begin{proof}
Note that $V_k\in \cM$ and $V_k\in L_p(\cM,\tau)$.
Hence, $V_k\in L_2(\cM,\tau)$.

Let
$$g_{k,l}=(\fD f)\cdot\chi_{[2^{-k-1},2^{-k})\times[2^{-l-1},2^{-l})}.$$
Let $\phi_{k,l}:\bR\rightarrow \bR $ be a compactly supported smooth function which vanishes near $0$ and such that $\phi_{k,l}=1$ on $[2^{-k-1},2^{-k})$ and on $[2^{-l-1},2^{-l}).$ Let $h:=f\phi_{k,l}.$ We have
$$g_{k,l}=(\fD h)\cdot\chi_{[2^{-k-1},2^{-k})\times[2^{-l-1},2^{-l})}.$$
Therefore, by Lemma \ref{lemma 3.6}, we have
\begin{align*}
T^{A,B}_{g_{k,l}}(V_kq_l)&=T^{A,B}_{g_{k,l}}(p_k(A-B)q_l)\\
&=T^{A,B}_{(\fD h)\cdot\chi_{[2^{-k-1},2^{-k})\times[2^{-l-1},2^{-l})}}(p_k(A-B)q_l)\\
&=(T^{A,B}_{(\fD h)})(T^{A,B}_{\chi_{[2^{-k-1},2^{-k})\times[2^{-l-1},2^{-l})}}(p_k(A-B)q_l))\\
&=T^{A,B}_{(\fD h)}(p_k(A-B)q_l)=p_k(h(A)-h(B))q_l=p_k(f(A)-f(B))q_l.
\end{align*}

Since $V_k \in L_2(\cM,\tau)$, it follows that
$$V_k=\sum_{l\geq k}V_kq_l,$$
where the series converges in $L_2$-topology (see e.g. \cite[Theorem 3.1]{DPS2016}).
Since $T^{A,B}_{g_k}$ is bounded on $L_2(\cM,\tau)$, it follows that
$$T^{A,B}_{g_k}(V_k)
=\sum_{l\geq k}T^{A,B}_{g_{k}}(V_kq_l)
=\sum_{l\geq k}T^{A,B}_{g_{k}}T^{A,B}_{\chi_{\bR \times [2^{-l-1},2^{-l})}}(V_kq_l)=
\sum_{l\geq k}T^{A,B}_{g_{k,l}}(V_kq_l),$$
where the series converges in $L_2$-topology.
 By the preceding paragraph,
$$T^{A,B}_{g_k}(V_k)=\sum_{l\geq k}p_k(f(A)-f(B))q_l,$$
where the series again converges in $L_2$-topology and therefore in local measure topology. Moreover,
$$\sum_{l\geq k}p_k(f(A)-f(B))q_l=p_k(f(A)-f(B))Q_k,$$
where the series converges locally in measure. The assertion follows by comparison of these $2$ equalities.
\end{proof}

The following lemma allows us to decompose the difference $f(A) -f(B)$.
For brevity, we say that a series $\sum_{k\in \bZ} X_{k}$, $X_k\in S(\cM,\tau)$, converges in local measure topology in the sense of  principal value if $\sum_{k=- n  }^{n} X_{k}$ converges in local measure topology.
We denote such convergence by $(p.v.) -\sum_{k\in \bZ}X_k$.
\begin{lem}\label{representation lemma}
 We have
\begin{align*}
s(A)_+\cdot (f(A)-f(B))\cdot s(B)_+
= (p.v.) -\sum_{k\in\mathbb{Z}}\Big( T^{A,B}_{g_k}(V_k)+T^{A,B}_{h_k}(W_k)\Big).
\end{align*}
\end{lem}
\begin{proof}
Observe that
\begin{align*}\sum_{k  =- n}^n\left(p_kQ_k+P_{k+1}q_k \right)
&=
\sum_{k  =- n}^n  p_k (Q_{-n }  - \sum_{l=-n}^{k-1} q_{l})   +  \sum_{k  =- n}^n (P_{n+1} +\sum_{l=k+1}^n p_l ) q_k \\
&=  \sum_{k  =- n}^n  p_k Q_{-n }  - \sum_{-n \le l < k \le n}p_k q_l   +  \sum_{k  =- n}^n P_{n+1} q_k    +\sum_{-n \le k <l \le n} p_lq _k \\
&=
\sum_{k  =- n}^n p_k Q_{-n }   + P_{n+1} \sum_{k  =- n}^n  q_k\\
&
 =
\chi_{[ 2^{-n-1}, 2^n )}(A) \chi_{(0,2^n)}(B)
+
\chi_{(0, 2^{-n-1}  )}(A) \chi_{[2^{-n-1},2^n)}(B).
\end{align*}
Noting  that $\chi_{[ 2^{-n-1}, 2^n )}(A)\rightarrow_n s(A)_+$,~
$\chi_{(0,2^n)}(B)\rightarrow_n s(B_+)$,~ $\chi_{(0, 2^{-n-1}  )}(A)\rightarrow_n 0$ and $\chi_{[2^{-n-1},2^n)}(B)\rightarrow _n 0$ in local measure topology (see e.g. \cite{DP2} and \cite[Chapter II, Section 7]{DPS}),
we obtain that
$\sum_{k  =- n}^n \left(p_kQ_k+P_{k+1}q_k \right)\rightarrow _n s(A)_+\cdot s(B)_+ $ in the local measure topology.
We write $$ s(A)_+ s(B)_+  =  (p.v.) - \sum_{ k \in \bZ}\left(p_kQ_k+P_{k+1}q_k \right) .$$
Therefore,
\begin{align*}s(A)_+f(A)s(B)_+=f(A)\cdot s(A)_+s(B)_+
&=f(A)\cdot(p.v.) -\sum_{k\in\mathbb{Z}} \Big(p_kQ_k+P_{k+1}q_k\Big)\\
&=(p.v.) -  \sum_{k\in\mathbb{Z}} \Big(p_kf(A)Q_k+P_{k+1}f(A)q_k\Big)
\end{align*}
and
\begin{align*}
s(A)_+f(B)s(B)_+ =s(A)_+s(B)_+\cdot f(B) &=\Big((p.v.) - \sum_{k\in\mathbb{Z}}p_kQ_k+P_{k+1}q_k\Big)\cdot f(B) \\
&=(p.v.) - \sum_{k\in\mathbb{Z}}\Big(p_kf(B)Q_k+P_{k+1}f(B)q_k\Big).\end{align*}
It follows that
$$s(A)_+\cdot(f(A)-f(B))\cdot s(B)_+=(p.v.) - \sum_{k\in\mathbb{Z}} \Big(p_k(f(A)-f(B))Q_k+P_{k+1}(f(A)-f(B))q_k\Big).$$

It remains to note that
$$p_k(f(A)-f(B))Q_k \stackrel{\eqref{Remark:5.1}}{=}T^{A,B}_{g_k}(V_k),\quad P_{k+1}(f(A)-f(B))q_k \stackrel{\eqref{Remark:5.1}}{=} T^{A,B}_{h_k}(W_k).$$

\end{proof}

An alternative proof for the following special case in the setting of (discrete) finite von Neumann algebras was sketched  in \cite[case 1 of Theorem 3.2]{Ricard}  by the complex interpolation theorem for $L_p$, $0<p<1$  (see e.g. \cite{Xu1990} and \cite[Lemma 2.5]{PiRi}, see also \cite[Section 3.2]{PST}).
We present a straightforward and complete proof for $S(\cM,\tau)$ below.
\begin{lem}\label{n lemma}
  We have
  $$\left\|s(A)(f(A)-f(B)) n(B)\right\|_p\leq \left\| f \right\|_{S_{d,\theta}} \left\|| A-B|^{\theta}\right\|_{p}.$$
  and
$$\left\|n(A)(f(A)-f(B)) s(B)\right\|_p\leq \left\|f\right\|_{S_{d,\theta}}\left\||A-B|^{\theta}\right\|_{p}.$$
\end{lem}
\begin{proof}
We only prove the first inequality. The proof for the second inequality is exactly the same.
Note that $$s(A)(f(A)-f(B)) n(B)=s(A) f(A) n(B) - s(A)f(B) n(B) = f(A)n(B) .$$
By the Spectral Theorem, we have
 \begin{align*} |f(A)n(B) |^2=  n(B)|f(A)|^2  n(B) &\le n(B) \left(\sup _{t\ne 0}\frac{|f(t)|}{|t|^\theta} \right)^2 |A|^{2\theta}n(B)\\
 & = \left|\sup_{t\neq0}
\frac{|f(t)|}{|t|^{\theta}} |A|^{\theta}n(B)\right|^2 .
\end{align*}
By the operator-monotonicity (see e.g. \cite[Proposition 1.2]{DD1995}) of function $t\mapsto \sqrt{t}$, $t\in \bR^+$, that
\begin{align}\label{ineq:ff}
\|f(A)n(B)\|_p
\leq
\sup_{t\neq0}
\frac{|f(t)|}{|t|^{\theta}}\cdot
\left\||A|^{\theta}n(B)\right\|_p.
\end{align}
The Araki-Lieb-Thirring inequality (see  \cite[Lemma 2.5]{Han}, see also \cite{Kosaki92})
states that
$$h(X^\theta Z^{2\theta } X^\theta )\prec\prec h((XZ^2 X)^\theta ), ~ 0\le X,Z\in S(\cM,\tau),$$
 where $h:[0,\infty)\to [0,\infty)$ is a continuous increasing function such that $t\mapsto h(e^t)$ is convex and $h(0)=0$.
 Setting $h(t) := t^{p/2}$, $t\ge 0$, we infer
 $$|Z^\theta X^\theta|^p = |X^\theta  Z^{2\theta} X^\theta |^{p/2}\prec\prec |XZ^2X|^{\theta p/2} = |ZX|^{\theta p}. $$
 Setting $Z:=|A|$ and $X:=n(B)$, we obtain that, we obtain that
$$ \left(\mu\left(  |A|^{ \theta}n(B)\right)\right)^p
\prec\prec\left(\mu\left(A n(B)\right)\right)^{p \theta} .$$
Hence,  $\left\||A|^\theta n(B) \right\|_p
 \le
 \left\| |An(B)|^\theta \right\|_p $,
 which together with   \eqref{ineq:ff} and the fact that $Bn(B)=0$ implies that
\begin{align*}
\|f(A)n(B)\|_p\leq\sup_{t\neq0}\frac{|f(t)|}{|t|^{\theta}}
\cdot\| | An(B)|^\theta \|_p
&=
\sup_{t\neq0}\frac{|f(t)|}{|t|^{\theta}}\cdot
\left\| |(A-B)n(B)|^\theta \right\|_p \\
& \leq
\sup_{t\neq0}
\frac{|f(t)|}{|t|^{\theta}}\cdot
\left \| |A-B|^\theta \right\|_p .
\end{align*}
\end{proof}

Before proceeding to the main result, we prove  the  following scalar inequality.
\begin{proposition}\label{5.5}
For every $0<\theta<1$ and $0<q<\infty$, there exists a $c_{\theta,q}>0$
such that
\begin{align}\label{obvious_ineq}
\sum_{l\in\mathbb{Z}}2^{q l(1-\theta)}\cdot\min\{\alpha,2^{1-l}\}^q \leq c_{\theta,q }\alpha^{\theta q }, ~\alpha \in \bR ^+ .\end{align}
\end{proposition}
\begin{proof}
Let $k$ be an integer be such that $\alpha \in (2^{-k},  2^{1-k}]$.
If $\alpha \le 2^{1-k}$, then $2^{k(1-\theta )} \le ( 2 \alpha^{-1} )^{1-\theta}$.
Noting that
$$\sum_{ l \le k } 2^{q l(1-\theta)} = 2^{q k(1-\theta )} \frac{1}{1- 2^{q(\theta-1)}}  ,$$
we have
\begin{align}\label{obvious1}
\sum_{ l \le k }  2^{q l (1-\theta)} \min\{\alpha,2^{1-l }\}^q
 &\le
\sum_{ l \le k }  2^{q l (1-\theta)} \alpha^q
=
2^{qk(1-\theta )} \frac{\alpha^q }{1- 2^{q( \theta-1) }} \nonumber \\
 &\le  ( 2 \alpha^{-1} )^{q(1-\theta)} \frac{ \alpha^q  }{1- 2^{q(\theta-1)}}
 =  \frac{2 ^{q(1-\theta)}  }{1- 2^{q(\theta-1)}}   \alpha^{\theta q}  .
\end{align}
On the other hand, since
$ \alpha > 2^{-k}  $,
it follows that
\begin{align*}
\sum_{l>k} 2^{ql(1-\theta)} \min\{\alpha , 2^{1-l}\}^q  &\le \sum_{l>k} 2^{q l(1-\theta)}  2^{q(1-l)}=  \sum_{l>k}   2^{q(1-l\theta) }\\
&= 2^q\cdot 2^{-qk\theta} \sum_{l>0} 2^{- q l\theta} =\frac{2^{q(1-\theta)}}{1- 2^{-q\theta} }2^{-qk\theta}\le \frac{2^{q(1-\theta)}}{1 -2^{-q\theta} } \alpha^{q\theta},
\end{align*}
which together with \eqref{obvious1}  implies \eqref{obvious_ineq}.
\end{proof}
The proof of the next lemma requires a conbimation of Proposition \ref{5.5}, Lemma \ref{representation lemma} and Theorem \ref{th:3.5}.
This is the  final intermediate step before our main result.
\begin{lem}\label{prefinal lemma}  Let $A,B\in S(\cM,\tau)$ be  such that $ A-B= x e$ for some $\tau$-finite projection $e$ and a real number $ x$.
There exists a constant $c_{p,\theta}$ such that
$$\left\|s(A)_+\cdot(f(A)-f(B))\cdot s(B)_+\right\|_p\leq c_{p,\theta}\left\|f\right\|_{S_{d,\theta}}\left\| |A-B|^\theta \right \|_p ,$$
$$\left\|s(A)_-\cdot(f(A)-f(B))\cdot s(B)_+\right\|_{p}\leq c_{p,\theta}\left\|f\right\|_{S_{d,\theta}}\left\| |A-B|^\theta \right\|_p  ,$$
$$\left\|s(A)_+\cdot(f(A)-f(B))\cdot s(B)_-\right\|_p\leq c_{p,\theta}\left\|f\right\|_{S_{d,\theta}}\left\| |A-B|^\theta \right\|_p  ,$$
$$\left\|s(A)_-\cdot(f(A)-f(B))\cdot s(B)_-\right\|_p\leq c_{p,\theta}\left\|f\right\|_{S_{d,\theta}}\left\| |A-B|^\theta \right\|_p  .$$
\end{lem}
\begin{proof}
We only prove the first inequality.
The proof of the  rest is exactly the same.

 It follows from Lemma \ref{representation lemma} that (see Notation \ref{nota})
\begin{align}\label{conv}
s(A)_+\cdot (f(A)-f(B))\cdot s(B)_+ =(p.v.)-\sum_{k\in\mathbb{Z}}\left(T^{A,B}_{g_k}(V_k)+T^{A,B}_{h_k}(W_k)\right) .
\end{align}
 Theorem \ref{th:3.5}   implies that there exists a constant $c_{p,\theta}$ such that that for every $k\in \bZ$,
$$\left\|T^{A,B}_{g_k}(V_k)\right\|_p\leq  c_{p,\theta}\cdot 2^{k(1-\theta)}\|f\|_{S_{d,\theta}}\|V_k\|_p.$$

Note that $\|V_k\|_\infty \le \|p_k A\|_\infty +\|BQ_k\|_\infty  \le 2^{1-k}$.
Hence,
$$\|V_k\|_p\le \|A-B\|_p =  \|A-B\|_\infty  \tau( e )^{ \frac 1  p } $$ and
$$
 \|V_k\|_{p } \le  \|V_k\|_\infty  \tau( \mbox{supp} (V_k))^{ \frac 1   p }  \le  \|V_k\|_\infty  \tau( e)^{ \frac 1   p } =  2^{1-k}   \tau( e )^{\frac 1  p }  .$$
There exists a constant $C_{p,\theta}$ such that
\begin{align*}
&\qquad (p.v.)-\sum_{k\in\mathbb{Z}} \left\|T^{A,B}_{g_k}(V_k) \right\|_p^p\\
&~\leq~ (p.v.)- \sum_{k\in\mathbb{Z}}
 c_{p,\theta}^p \cdot 2^{p k(1-\theta)}\left\|f\right\|_{S_{d,\theta}}^p  \min\{  \left\|A-B\right\|_\infty , 2^{1-k}\}^p  \tau( e )  \\ &\stackrel{\eqref{obvious_ineq}}{\le} C_{p,\theta} \left\|f\right\|_{S_{d,\theta}} ^p\left\|A-B\right\|^{p\theta} _\infty  \tau( e ) = C_{p,\theta} \|f\|_{S_{d,\theta}}^p  \left\| |A-B|^\theta \right\|_p^p .
 \end{align*}
The $p$-triangle inequality
 implies that $(p.v.)-\sum_{k\in\mathbb{Z}}T^{A,B}_{g_k}(V_k)$ converges in quasi-norm topology (similar for $(p.v.)- \sum_{k\in \bZ}T^{A,B}_{h_k}(W_k)$).
  A fortiori  these series converges in the local measure topology \cite{Sukochev}.
Hence, the series in \eqref{conv} converges to $s(A)_+\cdot (f(A)-f(B))\cdot s(B)_+$ in the $L_p$-topology.
Therefore,  there exists a constant $C'_{p,\theta}$ such that
$$\|s(A)_+\cdot (f(A)-f(B))\cdot s(B)_+\|_p \le C'_{p,\theta} \|f\|_{S_{d,\theta}} \left\||A-B|^{\theta} \right\|_p  ,$$
 which completes the proof.
\end{proof}


\begin{corollary}
Let $A,B\in S(\cM,\tau)$ be bounded and such that $A-B=xe$ for some $\tau$-finite projection $e$ and a real number $ x$.
There exists a constant $C_{p,\theta}$ such that
\begin{align}\label{5.6}
\left\|  f(A)-f(B)  \right\|_p \leq C_{p,\theta}\|f\|_{S_{d,\theta}}\left\| |A-B|^\theta \right\|_p.
\end{align}
\end{corollary}
\begin{proof}
Recall that $f(A)-f(B)$ can be rewritten as in the form of \eqref{eq:decomDF}.
The assertion follows from  Lemmas \ref{prefinal lemma} and \ref{n lemma}.
\end{proof}

\begin{lem}\label{final lemma}
Let $A,B\in S_h(\cM,\tau)$, with $A-B = \sum_{k=1} ^n  x_k e_k$,
where $1\le n<\infty$,  $x_k $ are real numbers and $e_k$ are mutually  orthogonal projections.
Then, there exists a constant $C_{p,\theta}$ such that
$$\left\|f(A) -f(B) \right\|_p \le C_{p,\theta} \|f\|_{S_{d,\theta}} \left\| |A-B|^{\theta} \right\|_{p}.$$
\end{lem}
\begin{proof}
Set $A_0:=B$ and
$$A_m:=B+\sum_{k=1}^{m } x_k e_k , ~  1 \le m\le n .$$
We have
$$f(A_n )-f(B)=\sum_{m=0}^{n-1} f(A_{m+1})-f(A_m).$$
Hence,   we have
\begin{align*}
\left\|f(A_n)-f(B)\right\|_p^p
\leq
\sum_{m=0}^{n-1} \left\|f(A_{m+1})-f(A_m) \right\|_p^p
&\stackrel{\eqref{5.6}}{\leq}
C_{p,\theta}^p\|f\|_{S_{d,\theta}}^p \sum_{m=0}^{n-1} \left\| |A_{m+1}-A_m|^\theta \right\|_p ^p \\
&~
=~C_{p,\theta}^p \|f\|_{S_{d,\theta}}^p\left\| |A-B|^\theta \right\|_p^p,
\end{align*}
which completes the proof.
\end{proof}

The latter lemma allows us to extend the result to the case when the difference $A-B$ (possibly unbounded) belongs to the noncommutative $L_p$-space, $p>0$.
A similar proof of the special  case of $t\mapsto t^\theta $ was given in  \cite[Theorem 3.4]{Ricard}.
We present a straightforward  proof below.
\begin{theorem}\label{th:main} Let $0<p\le\infty$ and $0< \theta<1$.
There exists $C_{p,\theta}$ such  that
  for
any semi-finite von Neumann algebra $(\mathcal{M},\tau)$,
and $A,\;B\in S(\mathcal{M},\tau)_h$ such that if $A-B\in L_{\theta p}(\mathcal{M},\tau)$ and any $f\in S_{d,\theta}$, then
$$
\left\| f(A)-f(B)\right\|_p\leq C_{p,\theta} \left\|f \right\|_{S_{d,\theta}} \left\| |A-B|^\theta \right\|_{p} .
$$
In particular, $f\in S_{\infty, \theta}$ is operator $\theta$-H\"older function with respect to  all $\left\|\cdot\right\|_p$, $p>0$.
\end{theorem}
\begin{proof}
We first assume that $0< p\le 1$.
Let $B-A=U|B-A|$ be the polar decomposition.
We define a sequence $\{K_n:= U\sum^{n^2}_{k=1} \frac{k-1}{n} E^{|B-A|}( \frac{k-1}{n},\frac{k}{n} ]\}$, which converges to $B-A$ in measure topology.
In particular, $A+K_n \stackrel{t_\tau}{\rightarrow} A+ B-A=B$ as $n\rightarrow \infty$.
By the continuity of functional calculus in $S(\cM,\tau)_h$,\footnote{See \cite{Tikhonov} or the comments below \cite[Proposition 2]{DP2}. See also the forthcoming book \cite[Chapter II, Theorem 8.7]{DPS}.}
we have $$f(A+K_n )\stackrel{t_\tau}{\rightarrow}  f(B) .$$
Applying
Lemma \ref{final lemma} to $A$ and $A+K_n$ and observing that $|K_n|\le | B-A |$, we obtain that
$$\left\|f(A) -f(A+K_n)\right\|_p  \le C_{p,\theta} \left\|f \right\|_{S_{d,\theta}}  \left \| | K_n|^\theta \right\|_p\le   C_{p,\theta} \left\|f \right\|_{S_{d,\theta}}  \left\| |A-B|^\theta  \right\|_p.$$
 Fatou's lemma (see e.g. \cite[Proposition 3.3]{DDP2} or \cite[Lemma 3.4]{FK}) implies that
\begin{align*}
\left\|f(A)-f(B)\right\|_p
&\leq\liminf_n \left\|f(A)-f(A+K_n )\right\|_p \le   C_{p,\theta} \left\|f \right\|_{S_{d,\theta}}  \left\| |A-B|^\theta  \right\|_p.
\end{align*}
The case when $p>1$ is a consequence of
 the case for $p =1 $ together with   Theorem \ref{th:subm}
 and the Hardy--Littlewood--Polya inequality \cite[Corollary 2.5]{HSZ}.
\end{proof}

\begin{remark}\label{non}
Although the  results in the present section  are  proved under the assumption that the von Neumann algebra $\cM$ acts on a separable Hilbert space,
 the above result indeed also holds for
non-separable Hilbert spaces.
\end{remark}
\begin{proof}[Sketch of proof]
Assume that $\cM$ is a general semifinite von Neumann algebra.
We first consider self-adjoint operators $X,Y\in L_{ \theta p }(\cM,\tau)$.
Note that 
  $\{ X_n :=   \sum_{k=1}^{n^2-1} \frac k n E^{X} [\frac k n, \frac{k+1}{n}  ) -\sum_{k=1}^{n^2-1} \frac k n E^{X} (-\frac{k+1}{n} , -\frac k n]  \}$ converges to $X$ and    $\{ Y_n :=  \sum_{k=1}^{n^2-1} \frac k n E^{Y} [\frac k n, \frac{k+1}{n}  )  -\sum_{k=1}^{n^2-1} \frac k n E^{Y} (-\frac{k+1}{n} , -\frac k n]  \} $ converges to $Y$.
 Consider the standard representation of $\cM$ on $L_2(\cM,\tau)$.
For every $n$,
$$\left\{ E^{X} (-n , -\frac {n^2-1} n]  ,\cdots,   E^{X} [\frac {n^2-1} n, n   ) ,  E^{X} (-n , -\frac {n^2-1} n]  ,\cdots, E^{Y} [\frac {n^2-1} n, n   )     \right\}$$
is a finite subset of $L_2(\cM,\tau)$, which generates a separable Hilbert  subspace of $L_2(\cM,\tau)$.
Hence, $$
\left\| f(X_n)-f(Y_n)\right\|_p \leq C_{p,\theta} \left\|f \right\|_{S_{d,\theta}} \left\| | X_n-Y_n|^{\theta}\right\|_{p}  .
$$
By the continuity of functional calculus and the Fatou Lemma, we obtain that
$$
\left\| f(X)-f(Y)\right\|_p \leq C_{p,\theta} \left\|f \right\|_{S_{d,\theta}}  \left\| |X-Y|^{\theta}\right\|_{p}  .
$$
The proof of \cite[Theorem 3.4]{Ricard} indeed allows us to extend the result to the general case when $X,Y\in S(\cM,\tau)$ with $X-Y\in L_{ \theta p}(\cM,\tau)$ and $\cM$ acts on a non-separable Hilbert space.
\end{proof}

\section{Operator $\theta$-H\"older  functions for $p$-th power of a symmetric space}\label{sec:symm}

Let $(\cM,\tau)$ be a semifinite von Neumann algebra represented on a Hilbert space.
In the following theorem, we obtain a   submajorization inequality  related to $\theta$-H\"{o}lder functions, which is the key tool in proving Corollary \ref{intr_main}.
Indeed, the following theorem holds under very general assumptions.
\begin{theorem}\label{th:subm}
Fix $p\in (0,\infty)$ and let  $g : \bR^+ \rightarrow \bR^+$ be a continuous increasing function.
Suppose that
$f:\bR\rightarrow \bC$ is a continuous  function   such that
 \begin{align}\label{6c1}
 \|f(X) -f(Y)\|_p \le C_{p,f,g}\|g(|X-Y|)\|_p,~ \forall X,Y\in S_h(\cM,\tau), ~g(|X-Y|)\in L_p(\cM,\tau)
 \end{align}
 and such that
 \begin{align}
 \label{6c2}\|f(X) -f(Y)\|_\infty  \le C_{\infty ,f,g}\|g(|X-Y|)\|_\infty, ~\forall  X,Y\in S_h(\cM,\tau), ~g(|X-Y|)\in \cM
 \end{align}
 for some constants $C_{p,f,g}$ and  $C_{\infty,f,g}$ depending on $g$, $p$ and $f$ only.
Then, there exists a constant $C_{p,f,g}'$ depending on $f, g$ and $p$, such that
for any $X,Y\in S_h(\cM,\tau)$, we have
$$\left(\mu(f(X)-f(Y))\right)^p \prec \prec C_{p,f,g} ' \mu(g(|X-Y|) )^p. $$
\end{theorem}
\begin{proof}
Without loss of generality, we may assume that $\cM$ is atomless (see e.g. \cite[Lemma 2.3.18]{LSZ}).
We first consider the case when $X-Y$ is $\tau$-compact.

For every $t>0$, we can find a projection $e_t\in \cM$ such that $\mu(s;X-Y)=\mu(s;|X-Y|e_t)$ for all $0\le s<t$ with $\tau(e_t)=t$ and $\|(X-Y)e^\perp_t\|_\infty \le \mu(t;X-Y) $ (see
\cite[Chapter III, Lemma 7.7]{DPS} or  \cite[p.953]{DDP1992}
).
Moreover, $e_t$ can be chosen to  commute with $X-Y$.
By Definition \ref{mu}, we have $\|(X-Y)e^\perp_t\|_\infty \ge \mu(t;X-Y) $. That is, $\|(X-Y)e^\perp_t\|_\infty = \mu(t;X-Y) $.
Let $$X_t:= Y+(X-Y)e_t .$$
Letting $c_p:= \max \{1,2^{p-1}\}$,
by \cite[Corollary 2.3.16]{LSZ}
and the fact that $(a+b)^p \le c_p (a^p+b^p)$ for any $a,b\ge 0$ (see \cite[(2.2) and (2.3)]{KPR}),
 we have
 $$\mu(f(X) -f(Y) )
\le  \mu(f(X_t) -f(Y))+ \left \| f(X) -f(X_t ) \right \|_\infty ,~ \forall t>0,   $$
 and therefore,  
\begin{align}\label{firststep}
\int_0^t \mu(s; f(X) -f(Y) )^p ds
&\le \int_0^t \Big( \mu(s; f(X_t) -f(Y)) + \left \| f(X) -f(X_t ) \right \|_\infty  \Big)^p ds\nonumber \\
&\le c_p\int_0^t \mu(s; f(X_t) -f(Y))^p ds +  c_p \int_0^t  \left \| f(X) -f(X_t ) \right \|_\infty^p ds \nonumber \\
&=   c_p\int_0^t \mu(s; f(X_t) -f(Y ))^p ds +  tc_p   \left \| f(X) -f(X_t ) \right \|_\infty^p \nonumber\\
& \le c_p\left\| f(X_t)  -f(Y) \right \|_p^p +   t c_p \left\| f(X) -f(X_t)\right\|_\infty^p.
\end{align}
Note that since $g$ is assumed to be monotone, we have
$\mu(g(|(X-Y ) e_t  |) )  =g(\mu((X-Y)e_t  ) )  =g(\mu(|X-Y|  ) )\chi_{(0,t)} $  and
$\left\| g( |(X  -  Y)e^\perp_t|)\right\|_\infty =  g(\mu(t;X  -  Y )  ) $ \cite[Corollary 2.3.17]{LSZ}.
Now, appealing to the hypothesis,  we obtain that
\begin{align*}
&\quad \int_0^t \mu(s; f(X) -f(Y) )^p ds \\
&\stackrel{\eqref{firststep}}{\le}
  c_p C_{p,f,g}^p  \left\|g(|X_t  -Y|) \right\|_{p}^{p } + c_p C_{\infty,f,g}^p t  \left\| g(|X  -  X_t | ) \right\|_\infty^{p} \\
& ~= ~ c_pC_{p,f,g}^p  \left\|g(|X-Y|e_t )\right\|_{p}^{p } +  c_pC_{\infty,f,g} ^p  t \left \| g(|X  -  Y|e^\perp_t )    \right \|_\infty^{p} \\
&~\le~  c_p(C_{p,f,g}^p +C_{\infty,f,g}^p ) \int_0^t \mu(s; g(|X  - Y|)) ^{p} ds,
\end{align*}
which completes the proof for the case when $X-Y$ is a $\tau$-compact operator.
For simplicity, we denote $C_1:=c_p(C_{p,f}^p +C_{\infty,f}^p )$.

Now, assume that $X-Y$ is not necessarily  $\tau$-compact.
Let
$$Z:=   (X-Y -\mu(\infty; X-Y)) _+  - (X-Y + \mu(\infty;X-Y))_-. $$
It is easy to see that $Z$ is $\tau$-compact (see e.g. \cite[Corollary 2.3.17 (d)]{LSZ}).
By the Spectral Theorem, we have
$$ \|X-Y-Z\|_\infty \le \mu(\infty;X-Y),$$
and
$$\mu(f(X) -f(Y)) \le \mu( f(Y+Z) -f(Y))  +  \left\| f(X) -f(Y+Z) \right\|_\infty .$$
Hence, by the result for $\tau$-compact operators, there exists a constant $C_{p,f,g}'$ such that
\begin{align*}
&\quad \int_0^t  \mu(s;f(X)-f(Y))^p ds \\
&\le c_p \int_0^t \mu(s;f(Y+Z)-f(Y))^p ds +c_p \int_0^t \left\| f(X)-f (Y+Z) \right\|_\infty^p ds  \\
&\le c_p C_1\int_0^t \mu(s;g(|Z|))^pds  + C_{\infty,f,g}\cdot  c_p \cdot  t  \left\| g(|X-Y-Z|)\right\|_\infty^p\\
&\le  c_p  C_1 \int_0^t \mu(s; g(|X-Y|))^p   ds + C_{\infty,f,g}  \cdot c_p \cdot   t \mu(\infty;g(|X-Y|))^p\\
&\le C_{p,f,g}'  \int_0^t  \mu(s;g(|X-Y|))^p ds,
\end{align*}
which completes the proof.

\end{proof}

\begin{remark}
By the above theorem, numerous results concerning operator  inequalities in the setting of noncommutative $L_p$-spaces, $p\ge 1$ (see e.g. \cite{AP2,CPSZ}), can be extended to the case of fully symmetric spaces.
\end{remark}

The main object of this section is the so-called $p$-th powers of symmetric spaces,
which play an important role in  analysis (see e.g. \cite{BHLS, ORS,Sukochev2016}).
Following the notation introduced
in  \cite{Xu} (see also \cite{DDP1992,ORS}), for $0<p<\infty$ and a quasi-Banach symmetric space $E(\cM,\tau)$,
 the \emph{$\frac 1p$-th power} of $E(\cM,\tau)$ is defined
by
\begin{align}\label{def:pcon}
 E(\cM,\tau)^{(p)} =\{  X\in S(\cM,\tau) : |X|^p \in   E(\cM,\tau)\} , \quad \|X\|_{  E^{(p)}} = \left\|\left|X\right|^p\right\|^{1/p}_  E.
 \end{align}
It is known (see e.g. \cite[Proposition 3.1]{DDS2014}) that $  E^{(p)}(\cM,\tau) =   E(\cM,\tau)^{(p)}$, where $  E^{(p)}(\cM,\tau)$ is the quasi-Banach symmetric space corresponding to the $\frac 1p$-th power  $  E(0,\infty)^{(p)}$ of the quasi-Banach symmetric function space $  E(0,\infty)$.
If $E(0,\infty)$ is a symmetric space, then  $E^{(p)}(\cM,\tau)$ is $p$-convex (see \cite[Proposition 3.1]{DDS2014}), that is,
there exists a constant $M$ such that for any finite sequence $\{X_k\}_{k=1}^n\subset E^{(p)}(\cM,\tau)$, we have
$$\left\|\left(\sum_{k=1}^n |X_k|^p \right)^{1/p} \right\|_{E^{(p)}}\le M \left(\sum_{k=1}^n \left\| X_k \right\|_{E^{(p)}} ^p  \right)^{1/p}. $$
If $E(0,\infty)$ is a fully symmetric space and  $0< p<\infty$,
then it is clear that for every $X\in E^{(p)}(\cM,\tau)  $ and  $Y\in S(\cM,\tau)$ with  $\mu(Y)^p \prec \prec \mu(X)^p$, we have  $Y\in E^{(p)} (\cM,\tau) $ with $\|Y\|_{E^{(p)}}\le \|X\|_{E^{(p)}}$.
We note that  $\left\|\cdot\right\|_{E^{(p)}}$ is a $p$-norm when $0< p\le 1$ (see e.g. \cite[Corollary 5.4]{DS2009} or \cite[Theorem 8.10]{Kalton_S}).
Most quasi-Banach symmetric spaces which occur in the
literature (such as $L_p$-spaces, $L_{p,q}$-spaces, etc.) can be constructed as the $1/p$-th power of a fully symmetric space.

Taking $g(s)=s^\theta$, $s\in \mathbb{R}^+$, $\theta<1$ and $f\in S_{d,\theta(p)}$, $0<p<\infty$.
Conditions \eqref{6c1} and \eqref{6c2} in   Theorem \ref{th:subm} are satisfied  (see Theorem \ref{th:main}
  and \cite[Theorem 4.1]{AP3}, respectively).
 By
  Theorem~\ref{th:subm}, we obtain  \eqref{intro_subm}, i.e.,  for any $X=X^*, Y=Y^*\in S(\cM,\tau)$, we have
  \begin{align}
\left(\mu\left(f(X) -f(Y)\right) \right)^p  \prec \prec C_{p,\theta}\|f\|_{S_{d,\theta}}^p \cdot \mu(|X-Y|^\theta)^p.
\end{align}
The following theorem is an immediate consequence of this inequality.
\begin{theorem}\label{th:mainE}For any fully symmetrically normed space $E(0,\infty)$, $0<p<\infty$, $0<\theta<1$ and $f\in S_{d,\theta}$, where $d=d(p)$,
there exists a constant $C_{p,\theta}$ such that for any $X,Y\in S(\cM,\tau)_h$ with $X-Y\in E^{(p)}(\cM,\tau)$,
we have
\begin{align}\label{eqmain}
\left\|f(X)-f(Y)\right\|_{E^{(p)}} \le C_{p,\theta} \left\|f \right\|_{S_{d,\theta}} \left\||X- Y|^\theta\right\|_{E^{(p)}}.\end{align}
\end{theorem}
Now, we consider two classic operator-monotone functions
 $t\mapsto \log(|t|+1)\in S_{\infty,\theta} $ and $t\mapsto \frac{|t|}{r+|t|} \in S_{\infty,\theta}$, where $r>0$, obtaining analogue of \cite[(ii) of Corollary 2]{Ando} (see also \cite[Theorem 3.4]{Kittaneh_K} for similar estimates).
\begin{corollary}
Assume that $E(0,\infty)$
is an arbitrary  fully symmetrically normed space and let
 $f(t):= \log (|t| +1)$ (or $f(t)={\rm sgn(t)}\log (|t| +1)$,  $\frac{|t|}{r+|t|}$, $\frac{t}{r+|t|}$, $r>0$), $t\in \bR$.
Then, for any $0<p<\infty$ and $0<\theta<1$, there exists a constant $C_{p,\theta}$ such that
\begin{align*}
\left\|f(X)-f(Y)\right\|_{E^{(p)}} \le C_{p,\theta} \left\|f \right\|_{S_{d,\theta}}  \left\||X- Y|^\theta\right\|_{E^{(p)}}, ~ X,Y\in S(\cM,\tau)_h. \end{align*}
\end{corollary}


For invertible functions, we have the following result (similar results have been obtained in \cite{Ando} and  \cite{AZ}).
\begin{corollary}\label{cor:inv}Let $\theta \in (1,\infty)$ and $p\in (0,\infty ) $.
If $f:\bR \rightarrow \bR \in S_{d,1/\theta}$ is invertible,
then there exists a constant $C_{p,\theta}$ such that
 $$ C_{p,\theta} \left\|f \right\|_{S_{d, 1/ \theta}}^p  | f^{-1}(X) -  f^{-1}(Y)|^{p/\theta}  \succ\succ  |X-Y|^{p} , ~X,Y\in S(\cM,\tau)_h .
$$
In particular, assuming $E(0,\infty)$ is a fully symmetrically normed space, we have
\begin{align}
\label{inver} C_{p,\theta}^{\theta/p} \left\|f \right\|_{S_{d, 1/ \theta}}^\theta \left\| f^{-1}(X) -  f^{-1}(Y) \right\|_{E^{(p)}}\ge
 \left\| (X-Y)^\theta \right\|_{E^{(p)}} .\end{align}
\end{corollary}
\begin{proof}
By \eqref{intro_subm}, we have
 $$ \left(\mu(f(X) -f(Y))\right)^p \prec \prec C_{p,\theta} \left\|f \right\|_{S_{d, 1/ \theta}} ^p \left(\mu (X-Y)\right)^{p/\theta}.$$
Substituting $X$ and $Y$ with $f^{-1}(X) $
and $f^{-1}(Y) $, we obtain that
$$  C_{p,\theta} \left\|f \right\|_{S_{d, 1/ \theta}}  ^p  \mu(f^{-1}(X) -  f^{-1}(Y))^{p/\theta} \succ\succ |X-Y|^p  .$$
Note that $\theta>1$.
 Inequality \eqref{inver} follows from the Hardy--Littlewood--Polya inequality  \cite[Corollary 2.5]{HSZ}.
\end{proof}
In particular, we obtain the following reverse inequality of \cite{Ricard}, which extends \cite[Corollary 4]{Ando} and \cite[Corollary 3 and Corollary 4]{Bhatia}.
Note that the following corollary holds even for operators in  $S(\cM,\tau)_h$ rather than positive operators as in  \cite{Ando} (see similar results in \cite[Corollaries 1 and 2]{AZ}).
We note  that the ``{\rm sgn}'' below can not be omitted (consider the case when  $X=-Y$).
\begin{corollary}\label{cor:>1}For any  $\theta\in (1,\infty)$,  $p\in (0,\infty)$ and  fully symmetrically normed space  $E(0,\infty)$, there exists a constant $C_{p,\theta}$ depending on $p$ and $\theta$ only such that
$$\left\| {\rm sgn}(X) |X|^\theta  - {\rm sgn}(Y) |Y|^\theta \right\|_{E^{(p)}}\ge C_{p,\theta }  \left\| |X-Y|^\theta \right\|_{E^{(p)}}, ~X,Y\in S_h(\cM,\tau), $$
and
$$\left\| {\rm sgn}(X) \left (e^{|X|} -1\right)  - {\rm sgn}(Y) \left (e^{|Y|} -1\right) \right\|_{E^{(p)}}\ge C_{p,\theta }  \left\| |X-Y|^\theta \right\|_{E^{(p)}}, ~X,Y\in S_h(\cM,\tau) . $$
\end{corollary}
\begin{proof}Appealing to Corollary \ref{cor:inv} with $g(t) = {\rm sgn}(t) |t|^\theta $ and $g(t)={\rm sgn}(t) \left (e^{|t|} -1\right)$.
\end{proof}

\section{Applications}

\subsection{Commutator and quasi-commutator estimates}
We consider  commutator and quasi-commutator estimates  for operator $\theta$-H\"{o}lder functions,
 which complement  \cite[Theorem 11.7]{AP2} and \cite[Theorem 10.5]{AP3}.
The proof of the following corollary  is essentially the same as the implication in \cite[Lemma 2.4]{Ricard2} via Cayley transform (for similar results  for  Lipschitz estimates, see \cite[Theorem 2.2]{DDPS}, \cite[Theorem 6.1]{CPSZ15} and \cite[Theorem 10.1]{AP3}).

\begin{corollary}\label{cor:com2}
Let $0<p< \infty$ and $0<\theta <1$.
Let   $f\in S_{d,\theta}$.
Let  $E(0,\infty)$ be a fully symmetrically normed space.
Then, there exists a constant $C_{p,\theta }$ such  that for $X\in S(\cM,\tau)_{h}$ and $B\in \cM$, we have
$$\left\|  [f(X),B]\right\|_{E^{(p)}} \le C_{p,\theta} \left\|f \right\|_{S_{d,\theta}} \left\| \left| [X,B] \right|^\theta \right\|_{E^{(p)}} \left\|B\right\|_\infty^{1-\theta}. $$
\end{corollary}
\begin{proof}
By  homogeneity, it suffices to prove the case when  $\|B\|_\infty =1$.
Let  $K_q:=\max\{2^{\frac{1}{q}-1},1\}$,
  which is the modulus of concavity of the quasi-norm $\left\|\cdot \right \|_{E^{(q)}}$, $q>0$ \cite[page 8]{Kalton}.
  We denote by ${\rm re}(B) $ and ${\rm im }(B)$ the real part and the imaginary part of $B$, respectively.
Note
$$   \left\|  [f(X),B]\right\|_{E^{(p)}} \le K_p\left( \left\|  [f(X),{\rm re}(B)]\right\|_{E^{(p)}}  +   \left\|  [f(X),{\rm im}(B)]\right\|_{E^{(p)}} \right) $$
and
\begin{align*}
&\quad \left\| \left| [X,{\rm re}(B)] \right|^\theta \right\|_{E^{(p)}}, ~\left\| \left| [X,{\rm im}(B)] \right|^\theta \right\|_{E^{(p)}}\\
 &\le ( \frac 1 2 K_{\theta p}\left\|  [X,B^*]  \right\|_{E^{(\theta p)}}
 + \frac 1 2 K_{\theta p}\left\|  [X,B]  \right\|_{E^{(\theta p)}})^\theta\\
&=  (  K_{\theta p}\left\|  [X,B ]  \right\|_{E^{(\theta p)}}
  )^\theta\\
  &=  K_{\theta p} ^\theta
  \left\|  \left|
  [X,B ] \right|^\theta \right\|_{E^{(  p)}},
\end{align*}
where we use the easy fact that  $$\mu([X,B])=\mu([X,B]^*) = \mu([ B^* , X] ) = \mu([X,B^*])$$ for the first equality.
We may assume that $B=B^*$.


Next, we use the Cayley transform defined by
\begin{align}\label{eq:UB}
U= (B-i)(B+i)^{-1}, ~B= 2i (1-U)^{-1}  -i.
\end{align}
Clearly, $U$ is unitary.
The functional calculus together with the assumption that $\|B\|_\infty =1$ yields that
\begin{align}\label{eq:U}
\left\|(1-U)^{-1} \right\|_\infty \le \frac {1}{\sqrt{2}} \mbox{ and } \left\|(B+i)^{-1} \right\|_\infty  \le 1.
\end{align}
Hence, we have
\begin{align*}
\left\|
             [f(X),B]
\right\|_{E^{(  p)}}
&~=
\left\|
           f(X)B -Bf(X)
\right\|_{E^{(  p)}}\\
& \stackrel{\eqref{eq:UB}}{=}
 2 \left\|f(X)(1-U)^{-1}- (1-U)^{-1} f(X) \right\|_{E^{(  p)}}  \\
&~\le  2 \left\|(1-U)^{-1}\right\|_\infty^2  \left\|f(X)(1-U)- (1-U) f(X) \right\|_{E^{(  p)}} \\
&\stackrel{\eqref{eq:U}}{=}   \left\|f(X) U -  U f(X) \right\|_{E^{(  p)}}\\
&~=   \left\|U^* f(X) U -   f(X) \right\|_{E^{(  p)}} \\
&~=   \left\|f(U^*  XU)  -   f(X) \right\|_{E^{(  p)}} \\
&\stackrel{\eqref{eqmain}}{\le   }  C_{p,\theta} \left\|f \right\|_{S_{d,\theta}} \left\||U^*  XU  -   X |^\theta\right\|_{E^{(  p)}} \\
& \stackrel{\eqref{def:pcon}}{=}   C_{p,\theta} \left\|f \right\|_{S_{d,\theta}} \left\|   XU  -   U X  \right\|_{E^{( \theta p)}}^{\theta }.
\end{align*}
By \eqref{eq:UB}, we have
\begin{align*}
[X,U]&= [X, (B-i)(B+i)^{-1}] = X (B-i)(B+i)^{-1} -  (B-i)(B+i)^{-1} X\\
&= (B+i)^{-1} \Big((B+i) X (B-i)-  (B-i) X (B+i)\Big) (B+i)^{-1} .
\end{align*}
Hence,
\begin{align*}
&\quad \left\|
             [f(X),B]
\right\|_{E^{(  p)}} \\
&\le  C_{p,\theta} \left\|f \right\|_{S_{d,\theta}}  \left\|(B+i)^{-1} \right\|_\infty^{2\theta}  \left\|(B+i) X (B-i) -  (B-i) X(B+i) \right\|_{E^{(  \theta p)}}^\theta \\
&\stackrel{\eqref{eq:U}}{\le}  C_{p,\theta} \left\|f \right\|_{S_{d,\theta}}   \left\| \left|2 [X,B]\right|^\theta   \right \|_{E^{(  p)}}.
\end{align*}
\end{proof}
A further consequence may be obtained for quasi-commutator estimates.
 The following corollary  extends \cite[Proposition 5.1]{Ricard} and \cite[Corollary 3.2]{Kittaneh_K}.
\begin{corollary}\label{cor:com3}
Let $0<p<  \infty$ and $0<\theta <1$.
Let   $f\in S_{d,\theta}$.
Let  $E(0,\infty)$  be a fully symmetrically normed space.
Then, there exists a constant $C_{p,\theta }$ such  that for $A,B\in S(\cM,\tau)_{h}$ and $R\in \cM$, we have
$$\left\|   f(A)R  - R f(B)\right\|_{E^{(p)}} \le C_{p,\theta} \left\|f \right\|_{S_{d,\theta}}  \left\| \left| AR  - R B \right|^\theta \right\|_{E^{(p)}} \left\|R \right\|_\infty^{1-\theta}. $$
\end{corollary}

\begin{proof}
By  homogeneity, it suffices to consider  the special case when  $\|R\|_\infty =1$.
Let us first assume that  the special case when $A,B$ are unitarily equivalent in $\cM$, i.e., $A=U^*BU$ for a unitary operator $U\in \cM$ and
we shall prove that
\begin{align}\label{eq:unit}
\left\|f(A) R -Rf(B)\right\|_{E^{(  p)}} &= \left\|U^* f(B)UR -Rf(B)\right\|_{E^{(  p)}} \nonumber \\
&\le  C_{p,\theta} \left\|f \right\|_{S_{d,\theta}}  \left\| U^* BUR -RB   \right\|_{E^{( \theta  p)}}^\theta.
\end{align}

By Corollary \ref{cor:com2}, we have
\begin{align*}
\left\|U^* f(B)UR -Rf(B)\right\|_{E^{(  p)}}& = \left\| f(B)UR -  U Rf(B)\right\|_{E^{(  p)}} \\
&\le C_{p,\theta} \left\|f \right\|_{S_{d,\theta}}  \left\|  BUR -  URB \right\|_{E^{(\theta   p)}}^\theta.
\end{align*}
Hence, we obtain the validity of \eqref{eq:unit}.

Now, we consider the case of arbitrary self-adjoint operators $A,B\in S(\cM,\tau)$.
We consider the algebra $\tilde{\cM}= \bM_2 \otimes \cM$ equipped with the trace $\tilde{\tau}=   Tr\otimes \tau$, where $Tr$ is the standard trace on $\bM_2$.
We consider $E(\tilde{\cM},\tilde{\tau})^{(p)}$ instead of $E(\cM,\tau)^{(p)}$.
For any $X,Y \in E(\cM,\tau)_h$, we have
$$\tilde{Z}: = \left(\begin{array}{cccccc}
X & 0  \\
0 &  Y \\
\end{array}\right) \in E(\tilde{\cM},\tilde{\tau})$$
with
\begin{align}\label{tilde}
\left\|X\right\|_ {E^{(p)}} \le  \left\|\tilde{Z} \right\|_{E^{(p)}} =\left\|\mu(X\oplus Y)\right\|_{E^{(p)}} \le K_{p} \left(\left\|X\right\|_{E^{(p)}} +\left\| Y\right\|_{E^{(p)}}  \right)  ,
\end{align}
where $K_{p}  $ stands for the modulus of concavity of $E^{(p)} (\cM,\tau)$.
Put
\begin{align*}
\tilde{A}=\left(\begin{array}{cccccc}
A & 0  \\
0 &  B \\
\end{array}\right),
\tilde{B}=\left(\begin{array}{cccccc}
B & 0  \\
0  & A \\
\end{array}\right),
~\mbox{ and }~
\tilde{R}=
\left(\begin{array}{cccccc}
R & 0  \\
0 &  R^* \\
\end{array}\right).
\end{align*}
Then, $ \tilde{A}$ and $\tilde{B}$ are unitarily equivalent in $\tilde{\cM}$.
We have
\begin{align*}
f(\tilde{A}) \tilde{R}=\left(\begin{array}{cccccc}
f(A ) R & 0  \\
0 &  f(B)  R^*\\
\end{array}\right)~ \mbox{ and }~
 \tilde{R} f(\tilde{B})  =\left(\begin{array}{cccccc}
 R f(B) & 0  \\
0  &   R^*  f(A) \\
\end{array}\right).
\end{align*}
Hence, by \eqref{tilde}, we have
$$\left\| f(A)R  - R f(B)\right\|_{E^{(  p)}}\le \left\| f(\tilde{A}) \tilde{R}-  \tilde{R} f(\tilde{B}) \right\|_{E^{(  p)}}    $$
and (noting that $A$ and $B$ are self-adjoint and therefore $(AR-RB)^*=R^*A -BR^*$)
\begin{align*}
\left\| \tilde{A} \tilde{R}-  \tilde{R} \tilde{B} \right\|_{E^{(  \theta p)}}&\le   K_{ \theta  p}  \left( \left\| AR  - RB \right\|_{E^{(  \theta  p)}}  +  \left\|BR^*  -R^*A \right\|_{E^{(  \theta  p)}}  \right)\\
&= 2 K_{\theta  p}  \left\| AR  - RB \right\|_{E^{(  \theta  p)}} ,
\end{align*}
which together with  \eqref{eq:unit} implies that
\begin{align*}\left\| f(A)R  - R f(B)\right\|_{E^{(  p)}}
&\le \left\| f(\tilde{A}) \tilde{R}-  \tilde{R} f(\tilde{B}) \right\|_{E^{(  p)}}
\stackrel{\eqref{eq:unit}}{\le} C_{p,\theta} \left\|f \right\|_{S_{d,\theta}}  \left\| \tilde{A} \tilde{R}-  \tilde{R} \tilde{B} \right\|_{E^{( \theta p)}}^\theta\\
&\le ( 2 K_{ \theta  p})^\theta   C_{p,\theta} \left\|f \right\|_{S_{d,\theta}}   \left\|      AR  - RB  \right\|_{E^{(  \theta p)}}^\theta  .\end{align*}
This  completes the proof.
\end{proof}

\begin{remark}
The proof in Corollary \ref{cor:com2} and Corollary \ref{cor:com3} indeed show that,  in the setting of a symmetrically quasi-normed  space,
 the commutator estimates, quasi-commutator estimates and the difference estimates for  operator $\theta$-H\"older functions are equivalent.
\end{remark}
\begin{remark}\label{7.4}
In Corollary \ref{cor:com3}, no restrictions on the supports of $f$ are needed. That is, we can consider  functions  having possibly  unbounded supports rather than functions having  compact supports as in \cite[Theorem 2.4]{Sobolev}.

Moreover,  in \cite[Theorem 2.4]{Sobolev},  the estimate of the quasi-norm of $[f(X),B]$ is obtained in terms of $\left|[X,B]\right|^{\sigma}$ only when $\sigma $ is strictly less than $\theta $, $\theta\in (0,1)$.
Moreover, the constant obtained in \cite[page 168]{Sobolev} goes to infinity when we take  $\sigma \uparrow\theta$.
However, in Corollary \ref{cor:com3}, we obtain the estimate in terms of $\left|[X,B]\right|^{\theta}$.

Note that the integer $d$ obtained in \cite[Theorem 2.4]{Sobolev} satisfies that $d >  \frac1p+\sigma    \in   (\frac1p, \frac1p+1) $, $p\in (0,1]$  and $\sigma \in (0,\theta)$ (one can take $d = \lfloor \frac1p+\sigma \rfloor +1 $), while the integer
$d(p)$ in our paper  is the minimal integer such that  $d(p) >\frac1p+2$ (see Section \ref{bounded}).
That is, $d(p)=  \lfloor\frac1p +2 \rfloor  +1 $.
There exist $0<p\le 1$,  $\theta $ and $\sigma$ such that $d(p)-1=  \lfloor\frac1p +2 \rfloor    = \lfloor \frac1p + \sigma \rfloor+1= \lfloor \frac1p + \theta \rfloor +1 $.
In other words, $d(p)$ is the minimal integer such that $d(p)>d= \lfloor \frac1p + \sigma \rfloor+1  $ but $d(p)$ does not
depend on $\sigma$ and $\theta$.

\end{remark}

\subsection{Estimates for absolute value map}
The estimates of the distance between the absolute values of two operators $A$ and $B$ have been obtained by several mathematicians (see e.g. \cite{Bhatia1988,Kosaki1984,Kittaneh_K}).
In particular, Kosaki \cite{Kosaki1984} proved that if $A,B \in (\cN
_*, \left\|\cdot \right\|_1)$ ($\cN_*$ stands for the predual of a general von Neumann algebra $\cN$), then
$$  \left\| |A|-|B|  \right\|_1  \le 2^{1/2}  \left(  \|A+B\|_1 \|A-B\|_1   \right)  .$$
This result was later extended by Kittaneh and Kosaki \cite{Kittaneh_K} to the von Neumann-Schatten $p$-class, $p\ge 2$, in $B(\cH)$:
$$  \left\| |A|-|B|   \right\|_p   \le \left( \left\| A+B \right\|_p  \left\| A-B \right \|_p  \right)^{1/2},~ A,B\in B(\cH).
$$
The case when $1\le p\le 2$ was proved by Bhatia \cite{Bhatia1988}:
$$  \left\| |A|-|B|   \right\|_p   \le  2^{1/p-1/2 }\left( \left\| A+B\right\|_p  \left\| A-B\right\|_p  \right)^{1/2},~ A,B\in B(\cH).
$$
 Bhatia \cite{Bhatia1988} also proved the estimates for an arbitrary  fully symmetric norm $\left\|\cdot \right\|$ on $B(\cH)$:
 $$  \left\| |A|-|B|   \right\|    \le  2^{ 1/2 }\left( \left\|A+B\right\|  \left\|A-B\right\|  \right)^{1/2},~ A,B\in B(\cH) .
$$
In this subsection,
we consider the quasi-norm estimates of the absolute value map in $S(\cM,\tau)$, which extends the results in \cite{Bhatia1988,Kosaki1984,Kittaneh_K}.

\begin{corollary}
Let $E(0,\infty)$ be a fully symmetrically normed space and let  $p\in (0,\infty)$.
There exists a constant $C_p$ such that
$$\left\| \left| A\right|-\left| B \right| \right\|_{E^{(p)}}  \le C_{p} \left( \left\|  A+B   \right\|_{E^{(p)}} \left\|  A-B \right\|_{E^{(p)}} \right)^{\frac12}, ~A,B\in S(\cM,\tau).$$
\end{corollary}
\begin{proof}
Without loss of generality, we may assume that $0< p\le 1$.
Applying Theorem \ref{th:mainE} to $f(t)= |t|^{\frac 12}, ~t\in \bR$, $|A|^2$ and $|B|^2$, we obtain that there exists a constant $C_p$ such that
\begin{align}\label{final:abs1}
\left\| |A|-|B| \right\|_{E^{(p)}}
\le
C_{p} \left\|  \left| |A|^2- |B|^2\right|^{\frac12}   \right\|_{E^{(p)}} , ~A,B\in S(\cM,\tau).
\end{align}
By \cite[Lemma  2.3.15]{LSZ}, there exists partial isometries $U$ and $V$  such that
\begin{align*}
 \left| |A|^2 -|B|^2 \right| =\left|
  \frac{2 A^*A   -2B^*B}{2}
\right| &=
\left|
  \frac{(A+B)^*(A-B)}{2}
+
  \frac{(A-B)^*(A+B)}{2}\right|\\
  &\le
 U\left|  \frac{(A+B)^*(A-B)}{2} \right|U^*
+
V\left|  \frac{(A-B)^*(A+B)}{2} \right|V^*
.
\end{align*}
By the monotonicity of function $t\mapsto t^{1/2}$, $t\in \bR^+$ (see e.g. \cite{DD1995}), we have \begin{align*}
\left| |A|^2 -|B|^2 \right|^{\frac 1 2}  \le
\left(
U\left|  \frac{(A+B)^*(A-B)}{2} \right|U^*
+
V\left|  \frac{(A-B)^*(A+B)}{2} \right|V^*
\right)^{\frac12}.
\end{align*}
Hence,
\begin{align}\label{final1}
&\quad
\left\| \left| |A|^2 -|B|^2 \right|^{\frac 1 2} \right\|_{E^{(p)}}^2 \nonumber\\
 &\le
 \left\|
\left(
U\left|  \frac{(A+B)^*(A-B)}{2} \right|U^*
+
V\left|  \frac{(A-B)^*(A+B)}{2} \right|V^*
\right)^{\frac12}
\right\|_{E^{(p)}} ^2 \nonumber \\
&=
 \left\|
\left(
U\left|  \frac{(A+B)^*(A-B)}{2} \right|U^*
+
V\left|  \frac{(A-B)^*(A+B)}{2} \right|V^*
\right)
\right\|_{E^{(p/2)}}  \nonumber \\
&
\le
\left( \left\|
 \frac{(A+B)^*(A-B)}{2}
\right\|_{E^{(p/2)}}^{p/2}
+
\left\|
 \frac{(A-B)^*(A+B)}{2}
\right\|_{E^{(p/2 )}}^{p/2}
\right)^{2/p}.
\end{align}
Recall that
 \cite[Theorem 2.2]{DD95} (see also \cite[Corollary 1.13]{Hiai})
 $$ \mu(XY)^{p/2} \prec \prec  \mu(X)^{p/2} \mu(Y)^{p/2},~X,Y\in S(\cM,\tau)  . $$
By \cite[Theorem 1]{Sukochev2016} (see also \cite{DykS}. Note that the notation  of $E^{(p)}$ in \cite{Sukochev2016} is different from that in our paper), we have
\begin{align*}
\left\|
 (A+B)^*(A-B)
\right\|_{E^{(p/2)}}^{p/2} &=
\left\|
| (A+B)^*(A-B)|^{p/2}
\right\|_{E } \\
&\le
\left\|
 \mu(A+B)^{p/2} \mu(A-B)^{p/2}
\right\|_{E } \\
&\le  \left\|
\mu(A+B )^{p/2}
\right\|_{E^{(2)} } \left\|
\mu( A-B )^{p/2}
\right\|_{E^{(2)} }\\
&= \left\|
A+B
\right\|_{E^{(p)} }^{p/2}  \left\|
 A-B
\right\|_{E^{(p )} }^{p/2},
\end{align*}
which together with \eqref{final:abs1} and \eqref{final1} completes the proof.
\end{proof}

{\bf Acknowledgements}
The final version of this paper was completed while the first and the third
named authors were visiting Central South University in Changsha, during the January of 2019.
They gratefully acknowledge Yong Jiao and Dejian Zhou, for their
kind hospitality.

The first author acknowledges the support of University International Postgraduate Award (UIPA).
The   second   author was supported by the Australian Research Council  (FL170100052).
The third author was partly funded by a UNSW Scientia Fellowship.
The authors would like to thank Edward McDonald and Xiao Xiong for helpful discussions.

\affiliationone{
   J. Huang, F. Sukochev and D. Zanin\\
   School of Mathematics and Statistics,\\
    University of New South Wales,\\
     Kensington, 2052, NSW,
   Australia
   \email{jinghao.huang@unsw.edu.au\\
   f.sukochev@unsw.edu.au\\
   d.zanin@unsw.edu.au}}


\begin{thebibliography}{99}



\bibitem{AP1} A. B. Aleksandrov, V. V. Peller,
{\it Hankel and Toeplitz-Schur multipliers},
Math. Ann.  324  (2002) 277--327.

\bibitem{AP2} A. B. Aleksandrov, V. V. Peller,
{\it Functions of operators under perturbations of class $\mathbf{S}_p$},
J. Funct. Anal.  258 (2010) 3675--3724.

\bibitem{AP3} A. B. Aleksandrov, V. V. Peller,
{\it Operator H\"{o}lder-Zygmund functions},
Adv. Math., 224  (2010) 910--966.

\bibitem{AP11}
A. B. Aleksandrov, V. V. Peller,
{\it Estimates of operator moduli of continuity,}
J. Funct. Anal.  261 (2011) 2741--2796.

\bibitem{AP4}
A. B. Aleksandrov, V. V. Peller,
{\it Operator and commutator moduli of continuity for normal operators,}
Proc. London. Math. Soc. (2012) 821--851.

\bibitem{AP2016}
A. B. Aleksandrov, V. V. Peller,
{\it Operator Lipschitz functions} (Russian) Uspekhi Mat. Nauk 71 (2016)(430), 3--106; translation in Russian Math. Surveys 71 (2016) no. 4, 605--702

\bibitem{APPS}
A. B. Aleksandrov, V. V. Peller,  D. Potapov, F. Sukochev,
{\it Functions of normal operators under perturbations,}
Adv. Math. 226  (2011) 5316--5251.





\bibitem{Ando} T. Ando, {\it Comparison of norms $\left\| \left|f(A)-f(B) \right| \right \|$ and $\left\|f(\left|A-B\right|) \right\|$}, Math. Z. 197 (1988) 403--409.

\bibitem{Ando_H}
T. Ando, F. Hiai,
{\it Operator log-convex functions and operator means,}
Math. Ann. 350 (2011) 611--630.

\bibitem{AZ}
T. Ando, X. Zhan,
{\it Norm inequalities related to operator monotone functions,}
Math. Ann. 
(1999) 771--780.

\bibitem{Bennett_S} C. Bennett, R. Sharpley,  {\it Interpolation of Operators,}  London, Academic Press Inc., 1988.






\bibitem{Bhatia}
R. Bhatia,
{\it Some inequalities for norm ideals,}
Comm. Math. Phys. 111 (1987) 33--39.

\bibitem{Bhatia1988}
R. Bhatia,
{\it Perturbation inequalities for the absolute value map in norm ideals of operators,}
J. Operator Theory 19 (1988) 129--136.

\bibitem{BKS} M.S. Birman, L.S. Koplienko,  M.Z. Solomjak, {\it Estimates of the spectrum of a difference of fractional powers of selfadjoint operators,} Izv. Vys\v{s}. U\v{c}ebn. Zaved. Matematika,  3:154 (1975) 3--10.



\bibitem{BS1977}
  M.S. Birman, M.Z. Solomjak,
{\it Estimates of singular numbers of integral operators,}
Russian Math. Surveys 32:1 (1977) 15--89.


\bibitem{BS89}
  M.S. Birman, M.Z. Solomjak,
  {\it  Estimates for the difference of fractional powers of selfadjoint operators under unbounded perturbations},
  Zap. Nauchn. Sem. Leningrad. Otdel. Mat. Inst. Steklov. (LOMI) 178 (1989), Issled. Line\v{i}n. Oper. Teorii Funktsi\v{i}. 18, 120--145, 185; translation in
J. Soviet Math. 61 (1992), no. 2, 2018--2035.

\bibitem{BS03}
  M.S. Birman, M.Z. Solomjak,
{\it Double operator integrals in a Hilbert space,}
Integr. Equ. Oper. Theory 47 (2003) 131--168.

\bibitem{BHLS}
A. Ber, J. Huang, G. Levitina, F. Sukochev,
{\it Derivations with values in ideals of semifinite von Neumann algebras,}
J. Funct. Anal.  272  (2017) 4984--4997.


\bibitem{CMPS}
M. Caspers, S. Montgomery-Smith, D. Potapov, F. Sukochev,
{\it The best constants for operator Lipschitz functions on Schatten classes,}
J. Funct. Anal. 267 (2014) 3557--3579.

\bibitem{CPSZ15}
M. Caspers, D. Potapov, F. Sukochev, D. Zanin,
{\it Weak type estimates for the absolute value mapping,}
J. Operator Theory 73:2 (2015) 361--384.

\bibitem{CPSZ}
M. Caspers, D. Potapov, F. Sukochev, D. Zanin,
{\it Weak type commutator and Lipschitz estimates: resolution of the Nazarov-Peller conjecture,}
Amer. J. Math. to appear.


\bibitem{DK}
Yu.L. Daletskii, S.G. Krein,
{\it Integration and differentiation of functions of Hermitian operators and application to the theory of perturbations,}
Trudy Sem. Funktsion. Anal. Voronezh. Gos. Univ. 1 (1956) 81--105 (in Russian).

\bibitem{Davies}
E.B. Davies,
{\it Lipschitz continuity of functions of operators in the Schatten classes}, J. London Math. Soc. 37 (1988) 148--157.





\bibitem{Dixmier}
J. Dixmier,
{\it Les algebres d'operateurs dans l'Espace Hilbertien, 2nd ed.,}
Gauthier-Vallars, Paris, 1969.


\bibitem{DD1995}
P. Dodds, T. Dodds,
{\it On a submajorization inequality of T. Ando,}
Operator Theory: Advances and Applications, Vol. 75, 1995, Birkh\"{a}user Verlag Basel/Switzerland.

\bibitem{DD95}
P. Dodds, T. Dodds,
{\it Some aspects of the theory of symmetric operator spaces,}
Quaest. Math. 18 (1995) 47--89.


\bibitem{DDP1992}
P. Dodds, T. Dodds, B. de Pagter,
{\it Fully symmetric operator spaces,}
Integr. Equ. Oper. Theory 15 (1992) 942--972.

\bibitem{DDP2}
P. Dodds,  T. Dodds,  B. de Pagter,
{\it
Noncommutative K\"othe duality,}
 Trans. Amer. Math. Soc. 339 (1993)
717-750.

\bibitem{DDPS}
P. Dodds,  T. Dodds,  B. de Pagter, F. Sukochev
{\it
Lipschitz continuity of the absolute value and Riesz projections in symmetric operator spaces,}
J. Funct. Anal. 148 (1997)
28--69.


\bibitem{DDST} P. Dodds, T. Dodds, F. Sukochev, O. Tikhonov,
{\it A non-commutative Yosida-Hewitt theorem and convex sets of measurable operators closed locally in measure,}
    Positivity 9 (2005) 457--484.




\bibitem{DP2}
P. Dodds, B. de Pagter,
{\it Normed K\"{o}the spaces: A non-commutative viewpoint,}
Indag. Math. 25 (2014) 206--249.

\bibitem{DDS2014}
P. Dodds, T. Dodds, F. Sukochev,
{\it On $p$-convexity and $q$-concavity in non-commutative symmetric spaces,}
Integr. Equ. Oper. Theory 78 (2014) 91--114.

\bibitem{DPS2016}
P. Dodds, B. de Pagter, F. Sukochev,
{\it Sets of uniformly absolutely continuous norm in symmetric spaces of measurable operators,}
Trans. Amer. Math. Soc. 368 (6) (2016), 4315--4355.

\bibitem{DPS}
P. Dodds, B. de Pagter, F. Sukochev,
{\it Theory of noncommutative integration,}
unpublished manuscript.


\bibitem{DS2009}
P. Dodds, F. Sukochev,
{\it Submajorisation inequalities for convex and concave functions of sums of measurable operators,}
Positivity 13 (2009) 107--124.


\bibitem{DykS}
K. Dykema, A. Skripka,
{\it H\"older's inequality for roots of symmetric operator spaces,}
Studia Math. 228 (2015) 47--54.

\bibitem{DSZ}
K. Dykema, F. Sukochev, D. Zanin,
{\it A decomposition theorem in $II_1$-factor,}
J. reine angew. Math. 708 (2015) 97--114.

\bibitem{dWS}
B. de Pagter, H. Witvliet, F. Sukochev,
{\it Double operator integral,}
J. Funct. Anal. 192 (2002) 52--111.


\bibitem{dS}
B. de Pagter, F. Sukochev,
{\it Differentiation of operator functions in non-commutative $L_p$-spaces,}
J. Funct. Anal. 212 (2004) 28--75.

\bibitem{FK}
T. Fack, H. Kosaki,
{\it Generalized $s$-numbers of $\tau$-measurable operators,}
Pacific J. Math. 123
(1986)
269--300.


\bibitem{Far}
Yu.B. Farforovskaya,
{\it The connection of the Kantorovich-Rubinshtein metric for spactral resolutions of self-adjoint operators with functions of operators,}
Vestnik Leningrad. Univ. 19 (1968) 94--97 (in Russian).





\bibitem{Han}
Y. Han,
{\it On the Araki-Lieb-Thirring inequality in semi-finite von Neumann algebra,}
Ann. Funct. Anal. 7 (2016) 622--635.


\bibitem{Hiai}
F. Hiai,
{\it Matrix Analysis: Matrix Monotone Functions, Matrix Means, and Majorization}
 Interdiscip. Inform. Sci. 16 (2010)  139--248.

\bibitem{Hiai_N}
F. Hiai, Y. Nakamura,
{\it Distance between unitary orbits in von Neumann algebras,
with Appendix: Generalized Powers-St${\o}$rmer inequality} by H. Kosaki,
Pacific J. Math. 138 (1989) 259--294.

\bibitem{HLS2017}
J. Huang, G. Levitina, F. Sukochev,
{\it Completeness of symmetric $\Delta$-normed spaces of $\tau$-measurable operators,}
Studia Math. 237 (3) (2017) 201--219.


\bibitem{HSZ}
J. Huang, F. Sukochev, D. Zanin,
{\it Logarithmic submajorization and order-preserving isometries,}
 J. Funct. Anal. 278 (4) (2020) 108352.




\bibitem{Kalton}
N. Kalton,
{\it Quasi-Banach spaces. Handbook of the geometry of Banach spaces}, Vol. 2, 1099--1130, North-Holland, Amsterdam, 2003.

\bibitem{KPR}
N. Kalton, N. Peck, J. Rogers,
{\it An F-space Sampler},
 London Math. Soc. Lecture Note Ser., vol.89, Cambridge University Press, Cambridge, 1985.

\bibitem{Kalton_S}
N. Kalton, F. Sukochev,
\textit{Symmetric norms and spaces of operators},
J. reine angew. Math.  621 (2008) 81--121.





\bibitem{Kato}
T. Kato,
{\it Continuity of the map $S\rightarrow |S|$ for linear operator,}
Proc. Japan Acad., 49 (1973) 157--160.

\bibitem{KPSS}
E. Kissin, D. Potapov, V. Shulman, F. Sukochev,
{\it Operator smoothness in Schatten norms for functions of several variables: Lipschitz conditions, differentiability and unbounded derivations,}
Proc. London Math. Soc. 105 (2012) 661--702.

\bibitem{Kittaneh}
F. Kittaneh
{\it Inequalities for the Schatten p-norm IV,}
Comm. Math. Phys. 104 (1986) 581--585.



\bibitem{Kittaneh_K}
F. Kittaneh, H. Kosaki,
{\it Inequalities for the Schatten $p$-norm. V,}
Publ. RIMS  23 (1986) 433--443.

\bibitem{Kosaki1984}
H. Kosaki,
{\it On the continuity of map $\varphi \to |\varphi|$ from the predual of a $W^*$-algebra,}
J. Funct. Anal. 59 (1984) 123--131.

\bibitem{Kosaki92}
H. Kosaki,
{\it An inequality of Araki-Lieb-Thirring (von Neumann algebra case)},
Proc. Amer. Math. Soc. 114 (1992)
477--481.


\bibitem{KPS}
S. Krein, Y. Petunin, E. Semenov,
{\it Interpolation of linear operators,} Trans. Math. Mon.,  54, AMS, Providence, 1982.


\bibitem{LSS}
H. Leschke, A. Sobolev, W. Spitzer,
{\it Trace formulas for Wiener-Hopf operators with applications to entropies of free fermionic equilibrium states,}
J. Funct. Anal. 273 (2017) 1049--1094.

\bibitem{LSZ}
S. Lord, F. Sukochev, D. Zanin,
 {\it Singular traces: Theory and applications,}
de Gruyter Studies in Mathematics, 46, 2013.

\bibitem{MOA}
A. Marshall, I. Olkin, B. Arnold,
{\it Inequalities: theory of majorization and its applications, second edition,}
Springer series in statistics, Springer,  New York, 2011.


\bibitem{MS20}
E. McDonald, F. Sukochev,
{\it Lipschitz estimates in quasi-Banach Schatten ideals,}
https://arxiv.org/abs/2009.08069



\bibitem{Nelson}
E. Nelson,
{\it Notes on non-commutative integration,}
J. Funct. Anal.  15 (1974) 103--116.


\bibitem{NF}
L. Nikol$^\prime$skaya, Yu. Farforovskaya,
{\it H\"older functions are operator-H\"older},  Algebra i Analiz 22 (2010), 198--213 (Russian). English translation:
St. Petersburg Math. J. 22 (2011) 657--668.

\bibitem{ORS}
S. Okada, W. Ricker, E. S\'{a}nchez P\'{e}rez,
{\it Optimal domain and integral extension of operators acting in function spaces},
in: Operator Theory, Advances and applications vol. 180. Birkh\"{a}user (2008).


\bibitem{Peller1985}
V. Peller,
{\it Hankel operators in the theory of perturbations of unitary and self-adjoint operators,}
Funktsional. Anal. i Prilozhen.
19 (2) (1985) 37--51 (in Russian).
English transl.: Funct. Anal. Appl. 19 (1985) 111--126.


\bibitem{Peller1990}
V. Peller,
{\it Hankel operator in the perturbation theory of unbounded self-adjoint operators,}
in: Analysis and Partial Differential Equations, in: Lec. Notes Pure Appl math., Dekker. New York, 1990, pp. 529--544.

\bibitem{PiRi}
G. Pisier, \'E. Ricard,
{\it The non-commutative Khintchine inequalities for $0<p<1$},
J. Inst. Math.
Jussieu,  16 (2017) 1103--1123.

\bibitem{PX}
G. Pisier, Q. Xu,
{\it  Non-commutative $L^p$-spaces, in Handbook of the Geometry of Banach spaces,}
Vol. 2, pp. 1459--1517. North-Holland, Amsterdam, 2003.




\bibitem{PS2008}
D. Potapov, F. Sukochev,
{\it Lipschitz and commutator estimates in symmetric operator spaces,}
J. Operator Theory 59 (2008) 211--234.


\bibitem{PS2009}
D. Potapov, F. Sukochev,
{\it Unbounded Fredholm modules and double operator integrals,}
J. Reine Angew. Math. 626 (2009) 159--185.

\bibitem{PS1}
D. Potapov, F. Sukochev,
{\it Double operator integrals and submajorization,} Math. Model. Nat. Phenom., 5 (2010)  317--339.

\bibitem{PS2}
 D. Potapov, F. Sukochev,
{\it Operator-Lipschitz functions in Schatten-von Neumann classes,} Acta Math. 207 (2011)  375--389.

\bibitem{PST}
D. Potapov, F. Sukochev, A. Tomskova,
{\it On the Arazy conjecture concerning Schur multipliers on Schatten ideals,} Adv. Math. 268 (2015)  404--422.


\bibitem{Powers_S}
R. Powers, E. St{\o}rmer,
{\it Free states of the canonical anticommutation relations,}
Comm. Math. Phys. 16 (1970) 1--33.

\bibitem{Ricard2}
\'{E}. Ricard,
 {\it H\"{o}lder estimates for the noncommutative Mazur maps},
 Arch. Math. (Basel)
 104
 (2015)
 37--45.

\bibitem{Ricard}
\'{E}. Ricard,
{\it Fractional powers on noncommutative $L_p$ for $p<1$,}
Adv. Math. 333 (2018) 194--211.




\bibitem{Se}
I. Segal,
{\it A non-commutative extension of abstract integration},
Ann. Math. 57 (1953) 401--457.


\bibitem{ST}
A. Skripka, A. Tomskova,
{\it Multilinear operator integrals: Theory and Applications,}
 Lecture Notes in Mathematics, 2250. Springer, Cham,  2019.

\bibitem{Sobolev}
A. Sobolev,
{\it Functions of self-adjoint operators in ideals of compact operators,}
J. London Math. Soc. (2) 95 (2017) 157--176.

\bibitem{Sobolev17}
A. Sobolev,
{\it Quasi-classical asymptotics for functions of Wiener-Hopf operators: smooth versus non-smooth symbols,}
Geom. Funct. Anal. 27 (2017) 676--725.

\bibitem{Sugiura}
M. Sugiura,
{\it Unitary representations and Harmonic analysis,}
An introduction. Second edition. North--Holland Mathematical Library, 44. North-Holland Publishing Co., Amsterdam; Kodansha, Ltd., Tokyo, 1990. xvi+452 pp.


\bibitem{Sukochev}
F. Sukochev,
{\it Completeness of quasi-normed symmetric operator spaces,}
Indag. Math. 25 (2014) 376--388.

\bibitem{Sukochev2016}
F. Sukochev,
{\it H\"older inequality for symmetric operator spaces and trace property of K-cycles,}
Bull. London Math. Soc. 48 (2016) 637--647.





\bibitem{Tikhonov}
O. Ye Tikhonov,
{\it Continuity of operator functions in topologies connected with a trace on a von Neumann algebra,}
Izv. Vyssh. Uchebn. Zaved. Math.  (1987) 77--79.

\bibitem{Triebel2}
H. Triebel,
{\it Theory of function spaces, II,}
Birkh\"auser, Basel-Boston-Berlin, 1992.

\bibitem{Hemmen_A}
J.L. van Hemmen, T. Ando,
{\it An inequality for trace ideals,}
Comm. Math. Phys. 76 (1980) 143--148.

\bibitem{Xu1990}
Q. Xu,
{\it Applications du th\'{e}or\`{e}me de factorisation pour des fonctions \`{a} valeurs op\'{e}rateurs,}
Studia Math. 95 (1990) 273--292 (in French).

\bibitem{Xu}
Q. Xu,
{\it Analytic functions with values in lattices and symmetric space of measurable operators, }
Math. Proc. Camb. Phil. Soc. 109 (1991) 541--563.


\end{thebibliography}
\end{document}